\documentclass[final]{siamltex}
\usepackage[active]{srcltx}
\usepackage[all]{xy}
\usepackage{subfig}
\usepackage{hyperref}
\usepackage{amssymb,latexsym, amsmath}
\usepackage{mathrsfs}
\usepackage{graphics}
\usepackage{latexsym}
\usepackage{psfrag}
\usepackage[dvips]{graphicx}
\usepackage{color}
\usepackage{pifont,marvosym}
\usepackage{mathbbol}
\usepackage{enumerate}
\usepackage{enumitem}

\newtheorem{assumption}{Assumption}
\newtheorem{remark}{Remark}

\newcommand{\R}{\mathbb{R}}
\newcommand{\N}{\mathbb{N}}

\newcommand{\Id}{\mathbb{I}}

\newcommand{\gr}{\textrm{graph}}

\newcommand{\haus}{\mathcal{H}}
\newcommand{\ve}{\varepsilon}

\newcommand{\erre}{\mathbb{R}}
\newcommand{\enne}{\mathbb{N}}

\newcommand{\f}{\varphi}

\newcommand{\lip}{\textrm{Lip}}

\title{Optimal transport with branching distance costs and the obstacle problem}

\author{Fabio Cavalletti \thanks{SISSA, via Beirut 2, IT-34014 Trieste (ITALY)},({\tt cavallet@sissa.it}). }

\begin{document}

\maketitle

\begin{abstract}
We address the Monge problem in metric spaces with a geodesic distance: $(X,d)$ is a Polish space and 
$d_N$ is a geodesic Borel distance which makes $(X,d_N)$ a possibly branching geodesic space. 
We show that under some assumptions on the transference plan we can reduce the transport problem to transport problems along family of geodesics.
We introduce three assumptions on a given $d_{N}$-monotone transference plan $\pi$ which imply respectively:
strongly consistency of disintegration, continuity of the conditional probabilities of the first marginal 
and a regularity property for the geometry of chain of transport rays. 
We show that this regularity is sufficient for the construction of a transport map with the same transport cost of $\pi$.

We apply these results to the Monge problem in $\erre^{d}$ with smooth, convex and compact obstacle obtaining the existence of an optimal map provided
the first marginal is absolutely continuous w.r.t. the $d$-dimensional Lebesgue measure.
\end{abstract}



\bibliographystyle{plain}

\section{Introduction}
\label{S:intro}

This paper concerns the Monge minimization problem in metric spaces with geodesic structure:
given two Borel probability measure $\mu,\nu \in \mathcal{P}(X)$, where $(X,d)$ is a Polish space, i.e. complete and separable metric space, we study the minimization of the functional
\[ 
\mathcal{I}(T) = \int d_{N}(x,T(x)) \mu(dy)  
\]
where $T$ varies over all Borel maps $T:X \to X$ such that $T_{\sharp}\mu = \nu$ and 
$d_{N}$ is a Borel distance that makes $(X,d_{N})$ a possibly branching geodesic space.
We will apply the results to the obstacle problem: let $C \subset \erre^{d}$ be a convex set with $\partial C = M$ smooth, 
$(d-1)$-dimensional compact submanifold of $\erre^{d}$. 
Let $X = (\erre^{d} \setminus C) \cup M$, $\mu,\nu \in \mathcal{P}(X)$ and $d_{M}(x,y)$ be the infimum among  
all the Lipschitz curves in $X$ connecting to $x$ to $y$ of the euclidean length of such curves. We will prove the existence of a solution for
\[
\min_{T : T_{\sharp} \mu = \nu} \int d_{M}(x,T(x)) \mu(dx),
\] 
provided $\mu \ll \mathcal{L}^{d}$.



Before describing our investigation, we present a little bit of the existing literature referring to \cite{villa:topics} and \cite{villa:Oldnew} 
for a deeper insight into optimal transportation.

In the original formulation given by Monge in 1781 the problem was settled in $\erre^{d}$, 
with the cost given by the Euclidean norm and the measures $\mu, \nu$ were supposed to be absolutely 
continuous and supported on two disjoint compact sets.
The original problem remained unsolved for a long time. 
In 1978 Sudakov \cite{sudak} claimed to have a solution for any distance cost function induced 
by a norm: an essential ingredient in the proof was that if $\mu \ll \mathcal{L}^{d}$ and $\mathcal{L}^{d}$-a.e. $\erre^{d}$ can be decomposed into 
convex sets of dimension $k$, then
then the conditional probabilities are absolutely continuous with respect to the $\haus^{k}$ measure of the correct dimension. 
But it turns out that when $d>2$, $0<k<d-1$ the property claimed by Sudakov is not true. An example with $d=3$, $k=1$ can be found in \cite{larm}.

The Euclidean case has been correctly solved only during the last decade. L. C. Evans and W. Gangbo in \cite{evagangbo} 
 solved the problem under the assumptions that 
$\textrm{spt}\,\mu \cap  \textrm{spt}\,\nu = \emptyset$,  $\mu,\nu \ll \mathcal{L}^{d}$ and their densities are Lipschitz function with compact support.
The first existence results for general absolutely continuous measures $\mu,\nu$ with compact support have been independently obtained by 
L. Caffarelli, M. Feldman and R.J. McCann in \cite{caffafeldmc} and by N. Trudinger and X.J. Wang in \cite{trudiwang}. 
Afterwards M. Feldman and R.J. McCann \cite{feldcann:mani} extended the results to manifolds with geodesic cost. 
The case of a general norm as cost function on $\erre^{d}$, including also the case with non strictly convex unitary ball, 
 has been solved first in the particular case of crystalline norm by L. Ambrosio, B. Kirchheim and A. Pratelli in 
\cite{ambprat:crist}, and then in fully generality by T. Champion and L. De Pascale in \cite{champdepasc:Monge}. 
The case with $(X,d_{N})$ non-branching geodesic space has been studied by S. Bianchini and the author in \cite{biacava:streconv}.

Concerning optimal transportation around a convex obstacle,
it is worth noting that the case with transport cost $d_{M}^{2}$ is studied in \cite{jimenez:obstacle}.

\subsection{Overview of the paper}
\label{Ss:over}

Let $(X,d_N)$ be a geodesic space, not necessarily Polish. To assure that standard measure theory can be used, 
there exists a second distance $d$ on $X$ which makes $(X,d)$ Polish and $d_N$ Borel on $X \times X$ with respect to the metric $d \times d$. 
We will prove that given a $d_N$-cyclically monotone transference plan $\pi \in \Pi(\mu,\nu)$,  
under appropriate assumptions on the first marginal and on the plan $\pi$, there exists an admissible 
map $T:X \to X$ with the same transference cost of $\pi$. 
Since we do not require $d_N$ to be l.s.c., 
the existence of an optimal transference plan is not guaranteed and our strategy doesn't rely on a possible optimality of $\pi$.
Moreover it is worth notice that due to the lack of regularity of $d_{N}$ we will not use the existence of optimal potentials $(\phi,\psi)$.

Our strategy to cope with the Monge problem with branching distance cost is the following: 
\begin{enumerate}
\item reduce the problem, via Disintegration Theorem, to transportation problems in sets where, 
under a regularity assumption on the first marginal and on $\pi$, we know how to produce an optimal map;
\item show that the disintegration of the first marginal $\mu$ on each of this sets verifies this regularity assumption;
\item find a transport map on each of these sets and piece them together.
\end{enumerate}

In the easier case of $d_{N}$ non-branching, given a $d_{N}$-cyclically monotone transference plan it is always possible 
to reduce the problem on single geodesics. The reduced problem becomes essentially one dimensional and there 
the precise regularity assumption is that the first marginal has no atoms (is continuous).

If $d_{N}$ is a branching geodesic distance this reduction can't be done anymore 
and there is not another reference set where the existence of Monge minimizer is known.  
The reduction set will be a concatenation of more geodesics and  to produce an optimal map 
we will need a regularity assumption also on the shape of this set.

As in the non-branching case, the reduction sets come from the class of geodesics used by a $d_{N}$-monotone plan $\pi$. 
This class can be obtained 
from a $d_{N}$-cyclical monotone set $\Gamma$ on which $\pi$ is concentrated: one can construct the set of transport rays $R$,
the transport set $\mathcal{T}_{e}$, i.e. the set of geodesics used by $\pi$, and from them construct
\begin{itemize}
\item the set $\mathcal T$ made of inner points of geodesics,
\item the set $a \cup b := \mathcal T_e \setminus \mathcal T$ of initial points $a$ and end points $b$.
\end{itemize}

Since branching of geodesics is admitted, $R$ is not a partition on $\mathcal{T}$. To obtain an equivalence relation 
we have to consider the set $H$ of chain of transport rays:
it is the set of couples $(x,y)$ such that we can go from $x$ to $y$ with a 
finite number of transport rays such that their common points are not final or initial points. 
 
Hence $H$ will provide the partition of the transport set $\mathcal{T}$ and each equivalence class, $H(y)$ for $y$ in the quotient space, will be a reduction set.

Even if a partition is given, the reduction to transport problems on the equivalence classes is not straightforward: a necessary and sufficient condition is that the disintegration of the measure $\mu$ 
w.r.t. the partition $H$ is  strongly consistent.
This is equivalent to the fact that there exists a $\mu$-measurable quotient map $f : \mathcal T \to \mathcal T$ of the equivalence relation induced by the partition. \\
Since this partition is closely related to the geodesics of $d_{N}$, the strong consistency will follow from 
a topological property of the geodesic as set in $(X,d)$ and from a metric property of $d_{N}$ as a function: 
\begin{enumerate}[label=(1.\alph*),ref=(1.\alph*)]
\item \label{I:comegeo} each chain of transport rays $H(y)$ restricted to a $d_{N}$ closed ball is $d$-closed; 
\item \label{I:comelocpt} $d_{N}(x,\cdot)$ restricted to $H(x)$ is bounded on $d$-bounded sets. 
\end{enumerate}
Observe that these conditions on $H$ and $d_{N}$ are the direct generalization of the ones on geodesics used in \cite{biacava:streconv}
(continuity and local compactness) and they depend on the particular choice of the transference plan.
This assumptions permit to disintegrate $\mu$ restricted to $\mathcal T$. 
Hence one can write
\[
 \mu\llcorner_{\mathcal{T}} = \int \mu_y m(dy), \quad m := f_\sharp \mu,  \quad \mu_y(f^{-1}(y)) = 1,
\]
i.e. the conditional probabilities $\mu_y$ are concentrated on the counterimages $f^{-1}(y)$ (which is an equivalence class). 
The reduced problems are obtained by disintegrating $\pi$ w.r.t. the partition $H \times (X \times X)$,
\[
\pi\llcorner_{\mathcal{T}\times \mathcal{T}} = \int \pi_y m(dy), \quad \nu = \int \nu_y m(dy) \quad \nu_y := (P_2)_\sharp \pi_y,
\]
and considering the problems on the sets $H(y)$ with marginals $\mu_y$, $\nu_y$ and cost $d_{N}$. 

To next step is study the continuity of the conditional probabilities $\mu_y$
and whether $\mu\llcorner_{\mathcal{T}_{e}}=\mu\llcorner_{\mathcal{T}}$ holds true.
To pursue this aim we consider a natural operation on sets: the translation along geodesics. If $A$ is a subset of $\mathcal T$, we denote by $A_t$ the set translated by $t$ in the direction determined by $\pi$. A rigorous definition of the translation of sets along geodesic will be given during the paper.
It turns out that $\mu(a \cup b) = 0$ and the continuity of $\mu_y$ both depend on how the function $t \mapsto \mu(A_t)$ behaves. 
Indeed assuming that: 
\begin{enumerate}[label=(2)]
\item \label{I:NDEatom} for all $A$ Borel there exists a sequence $\{t_{n}\}\subset \erre$ and $C>0$ such that $\mu(A_{t_{n}})\geq C \mu(A)$ as $t_{n}\to 0$,
\end{enumerate}
we have the following.

\begin{theorem}[Proposition \ref{P:puntini} and Proposition \ref{P:nonatoms}]
\label{T:-1}
If Assumption \ref{I:NDEatom} holds,
then $\mu(a \cup b) = 0$ and the conditional probabilities $\mu_y$ are continuous.
\end{theorem}

At this level of generality we don't know how to obtain a $d_{N}$-monotone admissible map 
for the restricted problem even if the marginal $\mu_{y}$ satisfies some regularity assumptions.
Therefore we need to assume that $H(y)$ has a particular structure: 

\begin{enumerate}[label=(3), ref=(3)] 
\item \label{I:clessidra3}  for $m$-a.e. $y$, the chain of transport rays $H(y)$ is contained, up to set of $\mu_{y}$-measure zero, in an uncountable ``increasing'' family of measurable sets. 
\end{enumerate}
A rigorous formulation of Assumption \ref{I:clessidra3} and of ``increasing'' will be given during the paper.
If $H(y)$ satisfies Assumption \ref{I:clessidra3}, then we can perform a disintegration of $\mu_{y}$ with respect to the partition induced 
by the uncountable ``increasing''  family of sets.
Then if the quotient measure and the marginal measures of $\mu_{y}$ are continuous, 
we prove the existence of an optimal map between $\mu_{y}$ and $\nu_{y}$.

\begin{theorem}[Proposition \ref{P:alternativa} and Theorem \ref{T:finale}]\label{T:0}
Let $\pi\in \Pi(\mu,\nu)$ be a $d_{N}$-monotone plan concentrated on a set $\Gamma$.
Assume that Assumptions \ref{I:comegeo}, \ref{I:comelocpt}, \ref{I:NDEatom}, \ref{I:clessidra3} holds and 
that the quotient measure and the marginal measures of $\mu_{y}$ are continuous for $m$-a.e. $y$.
Then  there exists an admissible map with the same transference cost of $\pi$.
\end{theorem}

It follows immediately that if we also assume that $\pi$ is optimal in the hypothesis of Theorem \ref{T:0}, then the Monge minimization problem admits a solution. 

Before presenting an application of Theorem \ref{T:0}, we summarize the theoretical results.
Let $\pi \in \Pi(\mu,\nu)$ be a $d_{N}$-cyclically monotone transference plan concentrated on a set $\Gamma$. 
We consider the corresponding family of chain of transport rays and, 
if assumptions \ref{I:comegeo} and \ref{I:comelocpt} are satisfied, we can perform, neglecting the set of initial points, a disintegration 
of $\mu,\nu$ and  $\pi$ with respect to the partition induced by the chain of transport rays.
Then if assumption \ref{I:NDEatom} is satisfied it follows that the set of initial points is $\mu$-negligible and 
the conditional probabilities $\mu_{y}$ are continuous. 
Since the geometry of $H(y)$ can be wild, we need another assumption to build a $d_{N}$-monotone transference map between $\mu_{y}$ 
and $\nu_{y}$. 
If $H(y)$ satisfies assumption \ref{I:clessidra3} we can perform another disintegration and, under additional regularity of 
the conditional probabilities of $\mu_{y}$ and of the quotient measure of $\mu_{y}$, we prove the existence of a 
$d_{N}$-monotone transference map between $\mu_{y}$ and $\nu_{y}$. 
Applying the same reasoning for $m$-a.e. $y$ we prove the existence of a transport map $T$ between $\mu$ and $\nu$ that has the 
same transference cost of the given $d_{N}$-cyclically monotone plan $\pi$.

In the last part of the paper we show an application of Theorem \ref{T:0}. 
Consider a hyper-surface $M \subset \erre^{d}$  that is the boundary of a convex and compact set $C$. 
Let $X$ be the closure, in the euclidean topology, of $\erre^{d} \setminus C$ and take as cost function 
$d_{M}$: the minimum of the euclidean length among all Lipschitz curves in $X$ that do not cross $M$. 
We will study the Monge minimization problem with $C$ as convex obstacle.
We will prove that if $\mu$ is absolutely continuous w.r.t. $\mathcal{L}^{d}$, the Monge minimization problem with cost $d_{M}$ admits a solution.

It is worth noting that the hypothesis of Theorem \ref{T:0}, namely Assumptions \ref{I:comegeo},
\ref{I:comelocpt}, \ref{I:NDEatom}, \ref{I:clessidra3}, are all about the behavior of a given $d_{N}$-cyclically monotone transference plan.  
For the obstacle problem we will prove that any $d_{N}$-cyclically monotone transference plan satisfies this hypothesis.
Therefore this concrete example provide also a confirmation of the validity of the proposed strategy and, in particular, of the 
non artificiality of the assumptions.

\subsection{Structure of the paper}
\label{Ss:structure}

The paper is organized as follows.

In Section \ref{S:preli}, we recall the mathematical tools we use in this paper. 
In Section \ref{Ss:univmeas} the fundamental results of projective set theory are listed. 
In Section \ref{S:disintegrazione} we recall the Disintegration Theorem. 
Next, the basic results of selection principles are listed in Section \ref{Ss:sele}, and in Section \ref{ss:Metric} we define the geodesic structure $(X,d,d_N)$ which is studied in this paper. Finally, Section \ref{Ss:General Facts} recalls some fundamental results in optimal transportation.

Section \ref{S:Optimal} shows how using only the $d_N$-cyclical monotonicity of a set $\Gamma$ we can obtain a partial order relation $G \subset X \times X$ as follows (Lemma \ref{L:analGR} and Proposition \ref{P:equiv}): $xGy$ iff there exists $(w,z) \in \Gamma$ and a geodesic $\gamma$, passing trough $w$ and $z$ and with direction $w \to z$, such that $x$, $y$ belongs to $\gamma$ and $\gamma^{-1}(x) \leq \gamma^{-1}(y)$. This set $G$ is analytic, and allows to define
\begin{itemize}
\item the transport rays set $R$ (\ref{E:Rray}),
\item the transport sets $\mathcal T_e$, $\mathcal T$ (with and without and points) (\ref{E:TR0}),
\item the set of initial points $a$ and final points $b$ (\ref{E:endpoint0}).
\end{itemize}
Even if this part  of Section \ref{S:Optimal} contains the same results of the first part of Section 3 of \cite{biacava:streconv}, 
for being as self contained as possible, we state this results and show their proofs again.
The main difference with the non-branching case is that here $R$ is not an equivalence relation. 
Therefore the approach proposed in \cite{biacava:streconv} doesn't work anymore and indeed 
the only common part with \cite{biacava:streconv} is  the first part of Section \ref{S:Optimal}.

To obtain an equivalence relation $H \subset X \times X$
we have to consider  the set of couples $(x,y)$ for $x,y\in \mathcal{T}$ such that 
there is a continuous path from $x$ to $y$, union of a
 finite number of transport rays never passing through $a\cup b$, Definition \ref{D:concat}. 
In Proposition \ref{P:equiv} we prove that $H$ is an equivalence relation.

Section \ref{S:partition} proves that the compatibility conditions 
\ref{I:comegeo} and \ref{I:comelocpt}
between $d_{N}$ and $d$ imply that the disintegration induced by $H$ on 
$\mathcal T$ is strongly consistent (Proposition \ref{P:sicogrF}). Using this fact we can reduce the analysis on $H(y)$ for $y$ in the quotient set.

In Section \ref{S:pt iniziali e finali} we prove Theorem \ref{T:-1}. We first introduce the operation $A \mapsto A_t$, the translation along geodesics (\ref{E:At}), and show that $t \mapsto \mu(A_t)$ is a $\mathcal{A}$-measurable function if $A$ is analytic (Lemma \ref{L:measumuAt}). \\ Next, we show that under the assumption
\[
\mu(A) > 0 \quad \Longrightarrow \quad  \mu(A_{t_{n}}) \geq C \mu(A)
\]
for an infinitesimal sequence $t_{n}$ and $C>0$,
the set of initial points $a$ is $\mu$-negligible (Proposition \ref{P:puntini}) and the conditional probabilities $\mu_y$ are continuous.

In Section \ref{S:Solution} we prove Theorem \ref{T:0}. 
First in Theorem \ref{T:finale} we prove that gluing all the $d_{N}$-cyclically monotone maps defined on $H(y)$ 
we obtain a measurable transference map $T$ from $\mu$ to $\nu$ $d_{N}$-cyclically monotone. 
Then the assumption on the structure of $\Gamma$ is stated  (Assumption \ref{A:clessidra3}) 
and in Proposition \ref{P:alternativa} we show that on the equivalence class $H(y)$ satisfying Assumption \ref{A:clessidra3}  there exists an 
optimal transference map $T_{y}$ from $\mu_{y}$ to $\nu_{y}$, provided the quotient measure and the marginal probabilities of $\mu_{y}$
induced by the partition given by Assumption \ref{A:clessidra3} are continuous.

Section \ref{S:application} gives an application of Theorem \ref{T:0}: $M$ is a connected smooth hyper-surface  of $\erre^{d}$ that
is the boundary of a convex and compact set $C$. Let $X=cl (\erre^{d}\setminus C)$. 
The distance $d_{M}$ is the minimum of the euclidean length among all the Lipschitz curves in $X$ (\ref{E:distostc}). 
Hence $C$ is to be intended as an obstacle for euclidean geodesics.
The geodesic space $(X,d_{M})$ fits into the setting of Theorem \ref{T:0} (Lemma \ref{L:cont} and Remark \ref{O:sempre}). 
If $\mu \ll \mathcal{L}^{d}$ then the $\mu$-measure of the set of initial points is zero and the marginal $\mu_{y}$ are continuous (Lemma \ref{L:appl}).  
Finally we show in Proposition \ref{P:fondamentale} and Proposition \ref{P:mappaostacolo1} that any $d_{M}$-cyclically monotone set and $\mu$ satisfy the hypothesis of Proposition \ref{P:alternativa}. 
It follows the existence of a solution for the Monge minimization problem.

\section{Preliminaries}
\label{S:preli}

In this section we recall some general facts about projective classes, the Disintegration Theorem for measure, measurable selection principles, geodesic spaces and optimal transportation problems.

\subsection{Borel, projective and universally measurable sets}
\label{Ss:univmeas}

The \emph{projective class $\Sigma^1_1(X)$} is the family of subsets $A$ of the Polish space $X$ for which there exists $Y$ Polish and $B \in \mathcal{B}(X \times Y)$ such that $A = P_1(B)$. The \emph{coprojective class $\Pi^1_1(X)$} is the complement in $X$ of the class $\Sigma^1_1(X)$. The class $\Sigma^1_1$ is called \emph{the class of analytic sets}, and $\Pi^1_1$ are the \emph{coanalytic sets}.

The \emph{projective class $\Sigma^1_{n+1}(X)$} is the family of subsets $A$ of the Polish space $X$ for which there exists $Y$ Polish and $B \in \Pi^1_n(X \times Y)$ such that $A = P_1(B)$. The \emph{coprojective class $\Pi^1_{n+1}(X)$} is the complement in $X$ of the class $\Sigma^1_{n+1}$.

If $\Sigma^1_n$, $\Pi^1_n$ are the projective, coprojective pointclasses, then the following holds (Chapter 4 of \cite{Sri:courseborel}):
\begin{enumerate}
\item $\Sigma^1_n$, $\Pi^1_n$ are closed under countable unions, intersections (in particular they are monotone classes);
\item $\Sigma^1_n$ is closed w.r.t. projections, $\Pi^1_n$ is closed w.r.t. coprojections;
\item if $A \in \Sigma^1_n$, then $X \setminus A \in \Pi^1_n$;
\item the \emph{ambiguous class} $\Delta^1_n = \Sigma^1_n \cap \Pi^1_n$ is a $\sigma$-algebra
and $\Sigma^1_n \cup \Pi^1_n \subset \Delta^1_{n+1}$.
\end{enumerate}
We will denote by $\mathcal{A}$ the $\sigma$-algebra generated by $\Sigma^1_1$: clearly $\mathcal{B} = \Delta^1_1 \subset \mathcal{A} \subset \Delta^1_2$.

We recall that a subset of $X$ Polish is \emph{universally measurable} if it belongs to all completed $\sigma$-algebras of all Borel measures on $X$:
it can be proved that every set in $\mathcal{A}$ is universally measurable.
We say that $f:X \to \erre \cup \{\pm \infty\}$ is a \emph{Souslin function} if $f^{-1}(t,+\infty] \in \Sigma^{1}_{1}$.

\begin{lemma}
\label{L:measuregP}
If $f : X \to Y$ is universally measurable, then $f^{-1}(U)$ is universally measurable if $U$ is.
\end{lemma}

See \cite{biacava:streconv} for the proof.

\subsection{Disintegration of measures}
\label{S:disintegrazione}

We follow the approach of \cite{biacar:cmono}.

Given a measurable space $(R, \mathscr{R})$ and a function $r: R \to S$, with $S$ generic set, we can endow $S$ with the \emph{push forward $\sigma$-algebra} $\mathscr{S}$ of $\mathscr{R}$:
$$
Q \in \mathscr{S} \iff r^{-1}(Q) \in \mathscr{R},
$$
which could be also defined as the biggest $\sigma$-algebra on $S$ such that $r$ is measurable. Moreover given a measure space 
$(R,\mathscr{R},\rho)$, the \emph{push forward measure} $\eta$ is then defined as $\eta := (r_{\sharp}\rho)$.

Consider a probability space $(R, \mathscr{R},\rho)$ and its push forward measure space $(S,\mathscr{S},\eta)$ induced by a map $r$. From the above definition the map $r$ is clearly measurable and inverse measure preserving.

\begin{definition}
\label{defi:dis}
A \emph{disintegration} of $\rho$ \emph{consistent with} $r$ is a map $\rho: \mathscr{R} \times S \to [0,1]$ such that
\begin{enumerate}
\item  $\rho_{s}(\cdot)$ is a probability measure on $(R,\mathscr{R})$, for all $s\in S$,
\item  $\rho_{\cdot}(B)$ is $\eta$-measurable for all $B \in \mathscr{R}$,
\end{enumerate}
and satisfies for all $B \in \mathscr{R}, C \in \mathscr{S}$ the consistency condition
$$
\rho\left(B \cap r^{-1}(C) \right) = \int_{C} \rho_{s}(B) \eta(ds).
$$
A disintegration is \emph{strongly consistent with r} if for all $s$ we have $\rho_{s}(r^{-1}(s))=1$. 
\end{definition}

We say that a $\sigma$-algebra $\mathcal{A}$ is \emph{essentially countably generated} with respect to a measure $m$, if there exists a countably generated $\sigma$-algebra $\hat{\mathcal{A}}$ such that for all $A \in \mathcal{A}$ there exists $\hat{A} \in \hat{\mathcal{A}}$ such that $m (A \vartriangle \hat{A})=0$.

We recall the following version of the theorem of disintegration of measure that can be found on \cite{Fre:measuretheory4}, Section 452.

\begin{theorem}[Disintegration of measure]
\label{T:disintr}
Assume that $(R,\mathscr{R},\rho)$ is a countably generated probability space, $R = \cup_{s \in S}R_{s}$ a decomposition of R, $r: R \to S$ the quotient map  
and $\left( S, \mathscr{S},\eta \right)$ the quotient measure space. Then  $\mathscr{S}$ is essentially countably generated w.r.t. $\eta$ and there exists a unique disintegration $s \to \rho_{s}$ in the following sense: 
if $\rho_{1}, \rho_{2}$ are two consistent disintegration then $\rho_{1,s}(\cdot)=\rho_{2,s}(\cdot)$ for $\eta-$a.e. $s$.
 
If $\left\{ S_{n}\right\}_{n\in \enne}$ is a family essentially generating 
$\mathscr{S}$ define the equivalence relation:
$$
s \sim s' \iff \   \{  s \in S_{n} \iff s'\in S_{n}, \ \forall\, n \in \enne\}.
$$
Denoting with p the quotient map associated to the above equivalence relation and with $(L,\mathscr{L}, \lambda)$ the quotient measure space, the following properties hold:
\begin{itemize}
\item $R_{l}:= \cup_{s\in p^{-1}(l)}R_{s} = (p \circ r)^{-1}(l)$ is $\rho$-measurable and $R= \cup_{l\in L}R_{l}$;
\item the disintegration $\rho = \int_{L}\rho_{l} \lambda(dl)$ satisfies $\rho_{l}(R_{l})=1$, for $\lambda$-a.e. $l$. In particular there exists a 
strongly consistent disintegration w.r.t. $p \circ r$;
\item the disintegration $\rho = \int_{S}\rho_{s} \eta(ds)$ satisfies $\rho_{s}= \rho_{p(s)}$, for $\eta$-a.e. $s$.
\end{itemize}
\end{theorem}

In particular we will use the following corollary.

\begin{corollary}
\label{C:disintegration}
If $(S,\mathscr{S})=(X,\mathcal{B}(X))$ with $X$ Polish space, then the disintegration is strongly consistent.
\end{corollary}

\subsection{Selection principles}
\label{Ss:sele}

Given a multivalued function $F: X \to Y$, $X$, $Y$ metric spaces, the \emph{graph} of $F$ is the set
\begin{equation}
\label{E:graphF}
\textrm{graph}(F) := \big\{ (x,y) : y \in F(x) \big\}.
\end{equation}
The \emph{inverse image} of a set $S\subset Y$ is defined as:
\begin{equation}
\label{E:inverseF}
F^{-1}(S) := \big\{ x \in X\ :\ F(x)\cap S \neq \emptyset \big\}.
\end{equation}
For $F \subset X \times Y$, we denote also the sets
\begin{equation}
\label{E:sectionxx}
F_x := F \cap \{x\} \times Y, \quad F^y := F \cap X \times \{y\}.
\end{equation}
In particular, $F(x) = P_2(\gr(F)_x)$, $F^{-1}(y) = P_1(\gr(F)^y)$. We denote by $F^{-1}$ the graph of the inverse function
\begin{equation}
\label{E:F-1def}
F^{-1} := \big\{ (x,y): (y,x) \in F \big\}.
\end{equation}

We say that $F$ is \emph{$\mathcal{R}$-measurable} if $F^{-1}(B) \in \mathcal{R}$ for all $B$ open. We say that $F$ is \emph{strongly Borel measurable} if inverse images of closed sets are Borel. A multivalued function is called \emph{upper-semicontinuous} if the preimage of every closed set is closed: in particular u.s.c. maps are strongly Borel measurable.

In the following we will not distinguish between a multifunction and its graph. Note that the \emph{domain of $F$} (i.e. the set $P_1(F)$) is in general a subset of $X$. The same convention will be used for functions, in the sense that their domain may be a subset of $X$.

Given $F \subset X \times Y$, a \emph{section $u$ of $F$} is a function from $P_1(F)$ to $Y$ such that $\textrm{graph}(u) \subset F$. We recall the following selection principle, Theorem 5.5.2 of \cite{Sri:courseborel}, page 198.

\begin{theorem}[Von Neumann]
\label{T:vanneuma}
Let $X$ and $Y$ be Polish spaces, $A \subset X \times Y$ analytic, and $\mathcal{A}$ the $\sigma$-algebra generated by the analytic subsets of X. Then there is an $\mathcal{A}$-measurable section $u : P_1(A) \to Y$ of $A$.
\end{theorem}

A \emph{cross-section of the equivalence relation $E$} is a set $S \subset E$ such that the intersection of $S$ with each equivalence class is a singleton. We recall that a set $A \subset X$ is saturated for the equivalence relation $E \subset X \times X$ if $A = \cup_{x \in A} E(x)$.

The next result is taken from \cite{Sri:courseborel}, Theorem 5.2.1.

\begin{theorem}
\label{T:KRN}
Let $Y$ be a Polish space, $X$ a nonempty set, and $\mathcal{L}$ a $\sigma$-algebra of subset of $X$. 
Every $\mathcal{L}$-measurable, closed value multifunction $F:X \to Y$ admits an $\mathcal{L}$-measurable section.
\end{theorem}

A standard corollary of the above selection principle is that if the disintegration is strongly consistent in a Polish space, then up to a saturated set of negligible measure there exists a Borel cross-section.

In particular, we will use the following corollary.

\begin{corollary}
\label{C:weelsupprr}
Let $F \subset X \times X$ be $\mathcal{A}$-measurable, $X$ Polish, such that $F_x$ is closed and define the equivalence relation $x \sim y \ \Leftrightarrow \ F(x) = F(y)$. Then there exists a $\mathcal{A}$-section $f : P_1(F) \to X$ such that $(x,f(x)) \in F$ and $f(x) = f(y)$ if $x \sim y$.
\end{corollary}

\begin{proof}
For all open sets $G \subset X$, consider the sets $F^{-1}(G) = P_1(F \cap X \times G) \in \mathcal{A}$, and let $\mathcal{R}$ be the $\sigma$-algebra generated by $F^{-1}(G)$. Clearly $\mathcal{R} \subset \mathcal{A}$.

If $x \sim y$, then
\[
x \in F^{-1}(G) \quad \Longleftrightarrow \quad y \in F^{-1}(G),
\]
so that each equivalence class is contained in an atom of $\mathcal{R}$, and moreover by construction $x \mapsto F(x)$ is $\mathcal{R}$-measurable.

We thus conclude by using Theorem \ref{T:KRN} that there exists an $\mathcal{R}$-measurable section $f$: this measurability condition implies that $f$ is constant on atoms, in particular on equivalence classes.
\end{proof}


\subsection{Metric setting}
\label{ss:Metric}

In this section we refer to \cite{burago}.

\begin{definition}
\label{D:lengthstr}
A \emph{length structure} on a topological space $X$ is a class $\mathtt{A}$ of admissible paths, which is a subset of all continuous paths in X, together with a map $L: \mathtt{A} \to [0,+\infty]$: the map $L$ is called \emph{length of path}. The class $\mathtt{A}$ satisfies the following assumptions:
\begin{description}
\item[closure under restrictions] if $\gamma : [a,b] \to X$ is admissible and $a \leq c\leq d \leq b$, then $\gamma \llcorner_{[c,d]}$ is also admissible.
\item[closure under concatenations of paths] if $\gamma : [a,b] \to X$ is such that its restrictions $\gamma_1, \gamma_2$ to $[a,c]$ and $[c,b]$ are both admissible, then so is $\gamma$.
\item[closure under admissible reparametrizations] for an admissible path $\gamma : [a,b] \to X$ and a for $\f: [c,d]\to[a,b]$, $\f \in B$, with $B$ class of admissible homeomorphisms that includes the linear one, the composition $\gamma(\f(t))$ is also admissible.
\end{description}

The map $L$ satisfies the following properties:
\begin{description}
\item[additivity] $L(\gamma \llcorner_{[a,b]}) = L(\gamma \llcorner_{[a,c]}) + L(\gamma \llcorner_{[c,b]})$ for any $c\in [a,b]$.
\item[continuity] $L(\gamma \llcorner_{[a,t]})$ is a continuous function of $t$.
\item[invariance] The length is invariant under admissible reparametrizations.
\item[topology] Length structure agrees with the topology of $X$ in the following sense: for a neighborhood $U_x$ of a point $x \in X$, the length of paths connecting $x$ with points of the complement of $U_x$ is separated from zero:
\[
\inf \big\{ L(\gamma) : \gamma(a)=x, \gamma(b) \in X\setminus U_{x} \big\} >0.
\]
\end{description}
\end{definition}

Given a length structure, we can define a distance
\[
d_N(x,y) = \inf \Big\{ L(\gamma): \gamma:[a,b]\to X, \gamma \in \mathtt{A}, \gamma(a)=x, \gamma(b)=y \Big\},
\]
that makes $(X,d_{N})$ a metric space (allowing $d_{N}$ to be $+\infty$). The metric $d_{N}$ is called \emph{intrinsic}. 
It follows from Proposition 2.5.9 of \cite{burago} that every admissible curve of finite length admits a constant speed parametrization, i.e.
$\gamma$ defined on $[0,1]$ and $L(\gamma\llcorner[t,t'])= v (t'-t)$, with $v$ velocity.

\begin{definition}
A length structure is said to be \emph{complete} if for every two points $x,y$ there exists an admissible path joining them whose length $L(\gamma)$ is equal to $d_{N}(x,y)$.
\end{definition}
Observe that in the previous definition we do no require $d_{N}(x,y)<+\infty$.

Intrinsic metrics associated with complete length structure are said to be \emph{strictly intrinsic}. The metric space $(X,d_N)$ with $d_N$ strictly intrinsic is called a \emph{geodesic space}. A curve whose length equals the distance between its end points is called \emph{geodesic}.

From now on we assume the following: \label{P:assumpDL}

\medskip
\begin{enumerate}
\item $(X,d)$ Polish space;
\item $d_N : X \times X \to [0,+\infty]$ is a  Borel distance;
\item $(X,d_N)$ is a geodesic space;
\end{enumerate}
\medskip

Since we have two metric structures on $X$, we denote the quantities relating to $d_N$ with the subscript $N$: for example
\[
B_r(x) = \big\{ y : d(x,y) < r \big\}, \quad B_{r,N}(x) = \big\{ y : d_N(x,y) < r \big\}.
\]
In particular we will use the notation
\[
D_N(x) = \big\{ y : d_N(x,y) < + \infty \big\},
\]
$(\mathcal{K},d_H)$ for the compact sets of $(X,d)$ with the Hausdorff distance $d_H$ and $(\mathcal{K}_N,d_{H,N})$ for the compact sets of $(X,d_N)$ with the Hausdorff distance $d_{H,N}$. We recall that $(\mathcal{K},d_H)$ is Polish.

\begin{lemma}
\label{L:measclos}
If $A$ is analytic in $(X,d)$, then $\{x : d_N(A,x) < \varepsilon \}$ is analytic for all $\varepsilon > 0$.
\end{lemma}

\begin{proof}
Observe that
\[
\big\{ x : d_N(A,x) < \varepsilon \big\} = P_1 \Big( X \times A \cap \big\{ (x,y) :d_N(x,y) < \varepsilon \big\} \Big),
\]
so that the conclusion follows from the invariance of the class $\Sigma^1_1$ w.r.t. projections.
\end{proof}

In particular, $\overline{A}^{_{d_N}}$, the closure of $A$ w.r.t. $d_N$, is analytic if $A$ is analytic.

\subsection{General facts about optimal transportation}
\label{Ss:General Facts}

Let $(X,\Omega,\mu)$ and $(Y,\Sigma,\nu)$ be two probability spaces and  $c:X \times Y \to \erre^{+}$ be a $\Omega \times \Sigma$ measurable function. Consider the set of transference plans
\[
\Pi (\mu,\nu) : =\Big\{ \pi \in \mathcal{P}(X\times Y) : (P_{1})_{\sharp}\pi = \mu, (P_{2})_{\sharp}\pi = \nu   \Big\},
\]
where $P_{i}(x_{1},x_{2})= x_{i}, i=1,2$. Define the functional
\begin{equation}
\label{E:Ifunct}
\begin{array}{ccccl} 
\mathcal{I} &:& \Pi(\mu,\nu) &\longrightarrow& \erre^{+} \cr
 & & \pi & \longmapsto &\displaystyle \mathcal{I}(\pi):=\int_{X\times Y} c \pi.
\end{array}
\end{equation}
The \emph{Monge-Kantorovich minimization problem} is the minimization of $\mathcal{I}$ over all transference plans.

If we consider a map $T : X \to Y$ such that $T_{\sharp}\mu=\nu$, the functional (\ref{E:Ifunct}) becomes
\[
\mathcal{I}(T):= \mathcal{I}( (Id \times T)_{\sharp}\mu ) = \int_{X} c(x,T(x)) \mu(dx).
\]
The minimization problem over all $T$ is called \emph{Monge minimization problem}.

The Kantorovich problem admits a (pre) dual formulation: before stating it, we introduce two definitions.

\begin{definition}
A map $\varphi : X \to \erre \cup \left\{ -\infty \right\} $ is said to be \emph{$c$-concave} if it is not identically $-\infty$ and there exists $\psi : Y \to \erre \cup \left\{ - \infty \right\}$, $\psi \not\equiv{-\infty}$, such that
\[
\varphi (x) = \inf_{y \in Y} \left\{ c(x,y) - \psi (y) \right\}.
\]
The \emph{$c$-transform} of $\varphi$ is the function
\begin{equation}
\label{E:ctransf}
\varphi^{c}(y):= \inf_{x\in X}  \left\{ c(x,y) - \varphi (x) \right\}.
\end{equation}
The \emph{$c$-superdifferential $\partial^{c}\f$} of $\varphi$ is the subset of $X \times Y$ defined by
\begin{equation}
\label{E:csudiff}
(x,y) \in \partial^{c}\f \iff\  c(x,y) - \f(x) \leq c(z,y) - \f(z) \quad \forall z \in X. 
\end{equation}
\end{definition}

\begin{definition}
A set $\Gamma \subset X\times Y$ is said to be 
\emph{c-cyclically monotone} if, for any $n \in \mathbb{N}$ and for any family $(x_{1},y_{1}), \dots, (x_{n},y_{n})$ of points of $\Gamma$, the following inequality holds
\[
\sum_{i=0}^{n}c(x_{i},y_{i}) \leq \sum_{i=0}^{n}c(x_{i+1},y_{i}),
\]
with $x_{n+1}=x_{1}$. A transference plan is said to be \emph{$c$-cyclically monotone} (or just \emph{$c$-monotone}) if it is concentrated on a $c$-cyclically monotone set.
\end{definition}

Consider the set
\begin{equation}
\label{E:Phicset}
\Phi_c := \Big\{ (\varphi,\psi) \in L^1(\mu) \times L^1(\nu): \varphi(x) + \psi(y) \leq c(x,y) \Big\}.
\end{equation}
Define for all $(\varphi,\psi)\in \Phi_c$ the functional
\begin{equation}
\label{E:Jfunct}
J(\varphi,\psi) := \int \varphi \mu + \int \psi \nu.
\end{equation}

The following is a well known result (see Theorem 5.10 of \cite{villa:Oldnew}).
\begin{theorem}[Kantorovich Duality]
\label{T:kanto}
Let X and Y be Polish spaces, let $\mu \in \mathcal{P}(X)$ and $\nu \in \mathcal{P}(Y)$, and let $c: X\times Y \to \erre^{+}\cup \left\{ +\infty \right\}$ be lower semicontinuous. 
Then the following holds:
\begin{enumerate}
\item Kantorovich duality:
\[
\inf_{\pi \in \Pi[\mu,\nu]} \mathcal{I} (\pi) = \sup _{(\varphi,\psi)\in \Phi_{c}} J(\varphi,\psi).
\]
Moreover, the infimum on the left-hand side is attained and the right-hand side is also equal to
\[
\sup _{(\varphi,\psi)\in \Phi_{c}\cap C_{b}} J(\varphi,\psi),
\]
where $C_{b}= C_{b}(X, \erre) \times C_{b}(Y,\erre)$.
\item If $c$ is real valued and the optimal cost 
\[
C(\mu,\nu):=\inf _{\pi \in \Pi(\mu,\nu)} I (\pi)
\]
is finite, then there is a measurable $c$-cyclically monotone set $\Gamma \subset X\times Y$, closed if $c$ is continuous, such that for any $\pi \in \Pi(\mu,\nu)$ the following statements are equivalent:
\begin{enumerate}
\item $\pi$ is optimal;
\item $\pi$ is $c$-cyclically monotone;
\item $\pi$ is concentrated on $\Gamma$;
\item there exists a $c$-concave function $\f$ such that $\pi$-a.s. $\f(x)+\f^{c}(y)=c(x,y)$.
\end{enumerate}

\item If moreover
\[
c(x,y) \leq c_{X}(x) + c_{Y}(y), \quad \ c_{X}\ \mu\textrm{-measurable}, \ c_{Y}\ \nu\textrm{-measurable},
\]
then there exist a couple of potentials and the optimal transference plan $\pi$ is concentrated on the set
\[
\big\{ (x,y) \in X\times Y\, | \, \varphi (x) + \psi (y) = c(x,y) \big\}.
\]
Finally if $(c_{X},c_{Y})\in \mathcal{L}^{1}(\mu) \times \mathcal{L}^{1}(\nu)$ then the supremum is attained 
\[
\sup_{\Phi_{c}} J = J(\f, \f^{c}).
\]
\end{enumerate}
\end{theorem}

We recall also that if $-c$ is Souslin, then every optimal transference plan $\pi$ is concentrated on a $c$-cyclically monotone set \cite{biacar:cmono}.

\section{Optimal transportation in geodesic spaces}
\label{S:Optimal}
Let $\mu, \nu \in \mathcal{P}(X)$ and consider the transportation problem with cost $c(x,y)= d_N(x,y)$, and let $\pi \in \Pi(\mu,\nu)$ be a $d_N$-cyclically monotone transference plan with finite cost. By inner regularity, we can assume that the optimal transference plan is concentrated on a $\sigma$-compact $d_N$-cyclically monotone set $\Gamma \subset \{d_N(x,y) < +\infty\}$. 

Consider the set
\begin{eqnarray}
\label{E:gGamma}
\Gamma' & :=&~ \bigg\{ (x,y) : \exists I \in \enne_0, (w_i,z_i) \in \Gamma \ \textrm{for}
 \ i = 0,\dots,I, \ z_I = y \crcr
&  &~ \qquad \qquad w_{I+1} = w_0 = x, \ \sum_{i=0}^I d_N(w_{i+1},z_i) - d_N(w_i,z_i) = 0 \bigg\}.
\end{eqnarray}
In other words, we concatenate points $(x,z), (w,y) \in \Gamma$ if they are initial and final point of a cycle with total cost $0$.

\begin{lemma}
\label{L:gGamma}
The following holds:
\begin{enumerate}
\item $\Gamma \subset \Gamma' \subset \{d_N(x,y) < +\infty\}$;
\item if $\Gamma$ is analytic, so is $\Gamma'$;
\item if $\Gamma$ is $d_N$-cyclically monotone, so is $\Gamma'$.
\end{enumerate}
\end{lemma}

\begin{proof}
For the first point, set $I=0$ and $(w_{n,0},z_{n,0}) = (x,y)$ for the first inclusion. If $d_N(x,y) = +\infty$, then $(x,y) \notin \Gamma$ and all finite set of points in $\Gamma$ are bounded.

For the second point, observe that
\begin{eqnarray*}
\Gamma' &=&~ \bigcup_{I \in \enne_0} P_{1,2I+1} (A_I) \crcr
&=&~ \bigcup_{I \in \enne_0} P_{1,2I+1} \bigg( \prod_{i=0}^I \Gamma \cap \bigg\{ \prod_{i=0}^I (w_i,z_i) : \sum_{i=0}^I d_N(w_{i+1},z_i) - d_N(w_i,z_i) = 0, w_{I+1} = w_0 \bigg\} \bigg).
\end{eqnarray*}
For each $I \in \N_0$, since $d_N$ is Borel, it follows that
\begin{eqnarray*}
\bigg\{ \prod_{i=0}^I (w_i,z_i) : \sum_{i=0}^I d_N(w_{i+1},z_i) - d_N(w_i,z_i) = 0, w_{I+1} = w_0 \bigg\}
\end{eqnarray*}
is Borel in $\prod_{i=0}^I (X \times X)$, so that for $\Gamma$ analytic each set $A_{n,I}$ is analytic. Hence $P_{1,2I+1}(A_I)$ is analytic, and since the class $\Sigma^1_1$ is closed under countable unions and intersections it follows that $\Gamma'$ is analytic.

For the third point, observe that for all $(x_j,y_j) \in \Gamma'$, $j=0,\dots,J$, there are $(w_{j,i},z_{j,i}) \in \Gamma$, $i = 0,\dots,I_j$, such that
\begin{eqnarray*}
d_N(x_j,y_j) + \sum_{i=0}^{I_j-1} d_N(w_{j,i+1},z_{j,i}) - \sum_{i=0}^{I_j} d_N(w_{j,i},z_{j,i}) = 0.
\end{eqnarray*}
Hence we can write for $x_{J+1} = x_0$, $w_{j,I_j+1} = w_{j+1,0}$, $w_{J+1,0} = w_{0,0}$
\begin{eqnarray*}
\sum_{j=0}^J d_N(x_{j+1},y_j) - d_N(x_j,y_j) =&~ \sum_{j=0}^J \sum_{i=0}^{I_j} d_N(w_{j,i+1},z_{j,i}) - d_N(w_{j,i},z_{j,i}) \geq 0,
\end{eqnarray*}
using the $d_N$-cyclical monotonicity of $\Gamma$.
\end{proof}

\begin{definition}[Transport rays]
\label{D:Gray}
Define the \emph{set of oriented transport rays}
\begin{equation}
\label{E:trG}
G := \Big\{ (x,y): \exists (w,z) \in \Gamma', d_N(w,x) + d_N(x,y) + d_N(y,z) = d_N(w,z) \Big\}.
\end{equation}

For $x \in X$, the \emph{outgoing transport rays from $x$} is the set $G(x)$ and the \emph{incoming transport rays in $x$} is the set $G^{-1}(x)$. Define the \emph{set of transport rays} as the set
\begin{equation}
\label{E:Rray}
R := G \cup G^{-1}.
\end{equation}
\end{definition}
The set $G$ is the set of all couples of points on oriented geodesics with endpoints in $\Gamma'$.
In $R$ the couples are non oriented.

\begin{lemma}
\label{L:analGR}
The following holds:
\begin{enumerate}
\item $G$ is $d_N$-cyclically monotone;
\item $\Gamma' \subset G \subset \{d_N(x,y) < +\infty\}$;
\item the sets $G$, $R := G \cup G^{-1}$ are analytic.
\end{enumerate}
\end{lemma}

\begin{proof}
The second point follows by the definition: if $(x,y) \in \Gamma'$, just take $(w,z) = (x,y)$ in the r.h.s. of (\ref{E:trG}).

The third point is consequence of the fact that
\[
G = P_{34} \Big( \big( \Gamma' \times X \times X \big) \cap \Big\{ (w,z,x,y) : d_N(w,x) + d_N(x,y) + d_N(y,z) = d_N(w,z) \Big\} \Big),
\]
and the result follows from the properties of analytic sets.

The first point follows from the following observation: if $(x_i,y_i) \in \gamma_{[w_i,z_i]}$, then from triangle inequality
\begin{eqnarray*}
d_N(x_{i+1},y_i) - d_N(x_i,y_i)  &\geq &~ d_N(x_{i+1},z_i) - d_N(z_i,y_i) - d_N(x_i,y_i) \crcr
&=&~ d_N(x_{i+1},z_i) - d_N(x_i,z_i) \crcr
&\geq&~  d_N(w_{i+1},z_i) - d_N(w_{i+1},x_{i+1})  - d_N(x_i,z_i) \crcr
&=&~  d_N(w_{i+1},z_i) - d_N(w_i,z_i) + d_N(w_i,x_{i}) - d_N(w_{i+1},x_{i+1}).
\end{eqnarray*}
Since $(w_{n+1},x_{n+1})=(w_{1},x_{1})$, it follows that
\[
\sum_{i=1}^{n} d_N(x_{i+1},y_i) - d_N(x_i,y_i) \geq \sum_{i=1}^{n} d_N(w_{i+1},z_i) - d_N(w_i,z_i) \geq 0.
\]
Hence the set $G$ is $d_N$-cyclically monotone.
\end{proof}

\begin{definition} Define the \emph{transport sets}
\begin{subequations}
\label{E:TR0}
\begin{eqnarray}
\label{E:TR}
\mathcal T :=&~ P_1 \big( G^{-1} \setminus \{x = y\} \big) \cap P_1 \big( G \setminus \{x = y\} \big), \\
\label{E:TRe}
\mathcal T_e :=&~ P_1 \big( G^{-1} \setminus \{x = y\} \big) \cup P_1 \big( G \setminus \{x = y\} \big).
\end{eqnarray}
\end{subequations}
\end{definition}

Since $G$ and $G^{-1}$ are analytic sets, $\mathcal{T}$, $\mathcal{T}_e$ are analytic. The subscript $e$ refers to the endpoints of the geodesics: clearly we have
\begin{equation}
\label{E:RTedef}
\mathcal{T}_e = P_1(R \setminus \{x = y\}).
\end{equation}

The following lemma shows that we have only to study the Monge problem in $\mathcal{T}_e$.

\begin{lemma}
\label{L:mapoutside}
It holds $\pi(\mathcal{T}_e \times \mathcal{T}_e \cup \{x = y\}) = 1$.
\end{lemma}

\begin{proof}
If $x \in P_1(\Gamma \setminus \{x=y\})$, then $x \in G^{-1}(y) \setminus \{y\}$. Similarly, $y \in P_2(\Gamma \setminus \{x=y\})$ implies that $y \in G(x) \setminus \{x\}$. Hence $\Gamma \setminus \mathcal{T}_e \times \mathcal{T}_e \subset \{x = y\}$.
\end{proof}

As a consequence, $\mu(\mathcal{T}_e) = \nu(\mathcal{T}_e)$ and any maps $T$ such that for $\nu \llcorner_{\mathcal{T}_e} = T_\sharp \mu \llcorner_{\mathcal{T}_e}$ can be extended to a map $T'$ such that $\nu = T_\sharp \mu$ with the same cost by setting

\begin{equation} \label{E:extere}
T'(x) = 
\begin{cases}
 T(x)  &  x \in \mathcal{T}_e  \\  x & x \notin \mathcal{T}_e 
\end{cases}
\end{equation}


\begin{definition}
\label{D:endpoint}
Define the multivalued \emph{endpoint graphs} by:
\begin{subequations}
\label{E:endpoint0}
\begin{eqnarray}
\label{E:endpointa}
a :=&~ \big\{ (x,y) \in G^{-1} : G^{-1}(y) \setminus \{y\} = \emptyset \big\}, \\
\label{E:endpointb}
b :=&~ \big\{ (x,y) \in G : G(y) \setminus \{y\} = \emptyset \big\}.
\end{eqnarray}
\end{subequations}
We call $P_2(a)$ the set of \emph{initial points} and $P_2(b)$ the set of \emph{final points}.
\end{definition}

\begin{proposition}
The following holds:
\begin{enumerate}
\item the sets
\[
a,b \subset X \times X, \quad a(A), b(A) \subset X,
\]
belong to the $\mathcal{A}$-class if $A$ analytic;
\item $a \cap b \cap \mathcal{T}_e \times X = \emptyset$;
\item $a(\mathcal{T}) = a(\mathcal{T}_e)$, $b(\mathcal{T}) = b(\mathcal{T}_e)$;
\item $\mathcal{T}_e = \mathcal{T} \cup a(\mathcal{T}) \cup b(\mathcal{T})$, $\mathcal{T} \cap (a(\mathcal{T}) \cup b(\mathcal{T})) = \emptyset$.
\end{enumerate}
\end{proposition}

\begin{proof}
Define
\[
C := \big\{ (x,y,z) \in \mathcal{T}_e \times \mathcal{T}_e \times \mathcal{T}_e: y \in G(x), z \in G(y) \big\} = (G \times X) \cap (X \times G) 
\cap \mathcal{T}_e \times \mathcal{T}_e \times \mathcal{T}_e,
\]
that is clearly analytic. Then
\[
b = \big\{ (x,y) \in G : y \in G(x), G(y) \setminus \{y\} = \emptyset \} = G \setminus P_{12}(C \setminus X \times \{y=z\}),
\]
\[
b(A) = \big\{ y: y \in G(x), G(y) \setminus \{y\} = \emptyset, x \in A \} = P_2(G \cap A \times X) \setminus P_2(C \setminus X \times \{y=z\}).
\]
A similar computation holds for $a$:
\[
a = G^{-1} \setminus P_{23}(C \setminus \{x=y\} \times X), \quad a(A) = P_1(G_{S} \cap X \times A) \setminus P_1(C \setminus \{x=y\} \times X).
\]
Hence $a,b \in \mathcal{A}(X \times X)$, $a(A), b(A) \in \mathcal{A}(X)$, being the intersection of an analytic set with a coanalytic one.
If $x \in \mathcal{T}_e \setminus \mathcal{T}$, then it follows that  $G(x)=\{x\}$ or $G^{-1}(x)=\{x\}$ hence $x \in a(x) \cup b(x)$.

The other points follow easily.
\end{proof}

\begin{definition}[Chain of transport rays]\label{D:concat}
Define the set of \emph{chain of transport rays}
\begin{eqnarray}
\label{E:partiz}
H &: = &~ \bigg\{ (x,y) \in \mathcal{T}_{e} \times \mathcal{T}_{e}: \exists I \in \enne_{0}, z_{i}\in \mathcal{T} \ \textrm{for}\ 1 \leq i\leq I, \crcr
&&~ \qquad \qquad  (z_{i},z_{i+1}) \in R, \ 0\leq i \leq I+1, \  z_{0}=x, \ z_{I+1}=y  \bigg\}.
\end{eqnarray}
\end{definition}
Using similar techniques of Lemma \ref{L:gGamma} it can be shown that  $H$ is analytic.

\begin{proposition}
\label{P:equiv}
The set $H \cap \mathcal{T} \times \mathcal{T}$ is an equivalence relation on $\mathcal{T}$. The set $G$ is a partial order relation on $\mathcal{T}_e$.
\end{proposition}

\begin{proof}
Using the definition of $H$, one has in $\mathcal{T}$:
\begin{enumerate}
\item $x \in \mathcal{T}$ clearly implies that $(x,x) \in H$;
\item since $R$ is symmetric, if $y \in H(x)$ then $x \in H(y)$;
\item if $y \in H(x)$, $z \in H(y)$, $x,y,z \in \mathcal{T}$. 
Glue the path from $x$ to $y$ to the one from $y$ to $z$. Since $y \in \mathcal{T}$, $z \in H(x)$.
\end{enumerate}

The second part follows similarly:
\begin{enumerate}
\item $x \in \mathcal{T}_e$ implies that
\[
\exists (x,y) \in \big( G \setminus \{x=y\} \big) \cup \big( G^{-1} \setminus \{x=y\} \big),
\]
so that in both cases $(x,x) \in G$;
\item $(x,y), (y,z) \in G \setminus \{x=y\}$ implies by $d_N$-cyclical monotonicity that $(x,z) \in G$.
\end{enumerate}
\end{proof}

We finally show that we can assume that the $\mu$-measure of final points and the $\nu$-measure of the initial points are $0$.

\begin{lemma}
\label{L:finini0}
The sets $G \cap b(\mathcal{T}) \times X$, $G \cap X \times a(\mathcal{T})$ is a subset of the graph of the identity map.
\end{lemma}

\begin{proof}
From the definition of $b$ one has that
\[
x \in b(\mathcal{T}) \quad \Longrightarrow \quad G(x) \setminus \{x\} = \emptyset,
\]
A similar computation holds for $a$.
\end{proof}

Hence we conclude that
\[
\pi (b(\mathcal{T}) \times X) = \pi(G \cap b(\mathcal{T}) \times X) = \pi(\{x = y\})
\]
and following (\ref{E:extere}) we can assume that
\[
\mu(b(\mathcal{T})) = \nu(a(\mathcal{T})) = 0.
\]

\section{Partition of the transport set}
\label{S:partition}
To perform a disintegration we have to assume some regularity of the support $\Gamma$ of the transport plan $\pi \in \Pi(\mu,\nu)$. 
From now on we will assume the following: \label{P:assumpDN}

\begin{assumption}\label{A:assu1}
We say that $\Gamma$ satisfies Assumption \ref{A:assu1} if
\begin{enumerate}[label=(\alph*), ref= (\alph*)]
\item \label{A:comegeo} for all $x\in \mathcal{T}$ and for all $r>0$ the set $H(x) \cap \overline{B_{r,N}(x)}^{d_N}$ is $d$-closed;
\item \label{A:comelocpt} for all $x\in \mathcal{T}$ there exists $r>0$ such that $d_{N}(x,\cdot)_{ \llcorner  H(x) \cap \overline{B_{r}(x)}} $
is bounded.
\end{enumerate}
\end{assumption}
Note that points \ref{A:comegeo} and \ref{A:comelocpt} of Assumption \ref{A:assu1} were already introduced at page \pageref{I:comegeo}.
Let $\{x_i\}_{i \in \N}$ be a dense sequence in $(X,d)$.


\begin{lemma}
\label{L:reguloclco}
The sets
\[
W_{ijk} := \Big\{ x \in \mathcal{T} \cap \bar B_{2^{-j}}(x_i): 
d_{N}(x,\cdot)_{\llcorner H(x)\cap \bar B_{2^{-j}}(x_i) } \leq k \Big\}
\]
form a countable covering of $\mathcal{T}$ of class $\mathcal{A}$.
\end{lemma}

\begin{proof}
We first prove the measurability. We consider separately the conditions defining $W_{ijk}$.

{\it Point 1.} The set
\[
A_{ij} := \mathcal{T} \cap \bar B_{2^{-j}}(x_i)
\]
is clearly analytic.

{\it Point 2.} 
The set
\[
D_{ijk}: = \bigg\{  (x,y) \in H : d(x_{i},y)\leq 2^{-j}, d_{N}(x,y) > k  \bigg\}
\]
is again analytic.
We finally can write
\begin{eqnarray*}
W_{ijk} = A_{ij}  \cap P_{1}( D_{ijk})^{c},
\end{eqnarray*}
and the fact that $\mathcal{A}$ is a $\sigma$-algebra proves that $W_{ijk} \in \mathcal{A}$.

To show that it is a covering, notice that from \ref{A:comelocpt} of Assumption \ref{A:assu1} for all $x \in \mathcal{T}$ there exists $r>0$ such that, on the set $H(x)\cap \bar B_{r}(x)$, 
$d_{N}(x,\cdot)$ is bounded. Choose $j$ and $i$ such that $2^{-j-1} \leq r$ and 
$d(x_{i},x)\leq 2^{-j-1}$, hence 
 \[
\bar B_{2^{-j}}(x_{i}) \subset \bar B_{r}(x)
\]
and therefore for some $\bar k \in \N$ we obtain that $x \in W_{ijk}$.
\end{proof}

\begin{remark}
\label{R:compFF}
Observe that $\bar B_{2^{-j}}(x_i) \cap H(x)$ is closed for all $x \in W_{ijk}$.
 
Indeed take $\{y_{n}\}_{n\in \enne} \subset \bar B_{2^{-j}}(x_i) \cap H(x)$ with $d(y_{n},y) \to 0$ as $n \to + \infty$, then since $x \in W_{ijk}$ 
it holds $d_{N}(x,y_{n}) \leq k$. By  \ref{A:comegeo} of Assumption \ref{A:assu1}, $d_{N}(x,y)\leq k$ and $y \in \bar B_{2^{-j}}(x_i) \cap H(x)$.
\end{remark}

\begin{lemma}
\label{L:partitilW}
There exist $\mu$-negligible sets $N_{ijk} \subset W_{ijk}$ such that the family of sets
\begin{eqnarray*}
\mathcal{T}_{ijk} = H^{-1}(W_{ijk} \setminus N_{ijk}) \cap\mathcal{T}
\end{eqnarray*}
is a countable covering of $\mathcal{T} \setminus \cup_{ijk} N_{ijk}$ into saturated analytic sets.
\end{lemma}

\begin{proof}
First of all, since $W_{ijk} \in \mathcal{A}$, then there exists $\mu$-negligible set $N_{ijk} \subset W_{ijk}$ such that $W_{ijk} \setminus N_{ijk} \in \mathcal{B}(X)$. Hence $\{W_{ijk} \setminus N_{ijk}\}_{i,j,k \in \N}$ is a countable covering of $\mathcal{T} \setminus \cup_{ijk} N_{ijk}$. It follows immediately that $\{\mathcal{T}_{ijk}\}_{i,j,k \in \N}$ satisfies the lemma.
\end{proof}

From any analytic countable covering, we can find a countable partition into $\mathcal{A}$-class saturated sets by defining
\begin{equation}
\label{E:Zkije}
\mathcal{Z}_{m} := \mathcal{T}_{i_mj_mk_m} \setminus \bigcup_{m' = 1}^{m-1} \mathcal{T}_{i_{m'}j_{m'}k_{m'}}, 
\end{equation}
where
\[
\N \ni m \mapsto (i_m,j_m,k_m) \in \N^3
\]
is a bijective map. 
Since $H$ is an equivalence relation on $\mathcal{T}$, we use this partition to prove the strong consistency.

On $\mathcal{Z}_m$, $m > 0$, we define the closed valued map
\begin{equation}
\label{E:mapTijkF}
\mathcal{Z}_m \ni x \mapsto F(x) := H(x) \cap \bar B_{2^{-j_m}}(x_{i_m}).
\end{equation}

\begin{proposition}
\label{P:sicogrF}
There exists a $\mu$-measurable cross section $f : \mathcal{T} \to \mathcal{T}$ for the equivalence relation $H$.
\end{proposition}

\begin{proof}
First we show that $F$ is $\mathcal{A}$-measurable: for $\delta > 0$,
\begin{eqnarray*}
F^{-1}(B_\delta(y)) =&~ \Big\{ x \in \mathcal{Z}_m: H(x) \cap B_{\delta}(y) \cap \bar B_{2^{-{j_m}}}(x_{i_m}) \not= \emptyset \Big\} \crcr
=&~ \mathcal{Z}_m \cap P_1 \Big( H \cap \big( X \times B_{\delta}(y) \cap \bar B_{2^{-{j_m}}}(x_{i_m}) \big) \Big).
\end{eqnarray*}
Being the intersection of two $\mathcal{A}$-class sets, $F^{-1}(B_\delta(y))$ is in $\mathcal{A}$. 
In Remark \ref{R:compFF} we have observed that $F$ is a closed-valued map, hence,
from Lemma 5.1.4 of \cite{Sri:courseborel}, $\gr(F)$ is $\mathcal{A}$-measurable.

By Corollary \ref{C:weelsupprr} there exists a $\mathcal{A}$-class section $f_m : \mathcal Z_m \to \bar B_{2^{-{j_m}}}(x_{i_m})$. The proposition follows by setting $f \llcorner_{\mathcal Z_m} = f_m$ on $\cup_m \mathcal Z_m$, and defining it arbitrarily on $\mathcal T \setminus \cup_m \mathcal Z_m$: the latter being
$\mu$-negligible, $f$ is $\mu$-measurable.
\end{proof}

Up to a $\mu$-negligible saturated set $\mathcal{T}_N$, we can assume it to have $\sigma$-compact range: just let $S \subset f(\mathcal{T})$ be a $\sigma$-compact set where $f_\sharp \mu$ is concentrated, and set
\begin{equation}
\label{E:TNngel}
\mathcal{T}_S := H^{-1}(S) \cap \mathcal{T}, \quad \mathcal{T}_N := \mathcal{T} \setminus \mathcal{T}_S, \quad \mu(\mathcal{T}_N) = 0.
\end{equation}

Hence we have a measurable cross-section
\[
\mathcal{S} := S \cup f(\mathcal{T}_N) = (\textrm{Borel}) \cup (f(\textrm{$\mu$-negligible})).
\]
Hence Disintegration Theorem \ref{T:disintr} yields 
\begin{equation}
\label{E:disiparzmu}
\mu\llcorner_{\mathcal{T}} = \int_S \mu_y m(dy), \quad m = f_\sharp \mu\llcorner_{\mathcal{T}}, \ \mu_y \in \mathcal{P}(H(y))
\end{equation}
and the disintegration is strongly consistent since the quotient map $f : \mathcal T \to \mathcal T$ is $\mu$-measurable and $(\mathcal T,\mathcal B(\mathcal T))$ is countably generated.

Observe that $H$ induces an equivalence relation also on $\mathcal{T} \times X \cap \Gamma$ where the equivalence classes are $H(y)\cap \mathcal{T} \times X$ and
the quotient map is the $f$ of Proposition \ref{P:sicogrF}. Hence  
\begin{equation}
\label{E:disiparzpi}
\pi\llcorner_{\mathcal{T} \times X \cap \Gamma} = 
\int_S \pi_y m_{\pi}(dy), \quad m_{\pi} = f_\sharp \pi\llcorner_{\mathcal{T} \times X \cap \Gamma}, \ \pi_y \in \mathcal{P}(H(y)\cap \mathcal{T} \times X).
\end{equation}
Observe that $m = m_{\pi}$.

\section{Regularity of the disintegration}
\label{S:pt iniziali e finali}

In this Section we consider the translation of Borel sets by the optimal geodesic flow, we introduce the fundamental regularity assumption (Assumption \ref{A:NDEatom}) on the measure $\mu$ and we show that an immediate consequence is that the set of initial points is negligible and consequently we obtain a disintegration of $\mu$
on the whole space. A second consequence is that the disintegration of $\mu$ w.r.t. the $H$ has continuous conditions probabilities.

\subsection{Evolution of Borel sets}
\label{Ss:evolution}

Let $A \subset \mathcal{T}_e$ be an analytic set and define for $t \in \R$ the \emph{$t$-evolution $A_t$ of $A$} by: 
\begin{equation}
\label{E:At}
A_t :=
\begin{cases}
P_{2}\big\{ (x,y)\in G \cap A \times X : d_{N}(x,y)=t \big\} & t\geq 0 \\
 P_{2}\big\{ (x,y)\in G^{-1} \cap A \times X : d_{N}(x,y)=t \big\} & t<0.
\end{cases}
\end{equation}

It is clear from the definition that if $A$ is analytic, also $A_t $ is analytic .
We can show that $t \mapsto \mu(A_t)$ is measurable.

\begin{lemma}
\label{L:measumuAt}
Let $A$ be  analytic. The function $t \mapsto \mu(A_t)$ is $\mathcal{A}$-measurable for $t \in \erre$. 
\end{lemma}

\begin{proof}
We divide the proof in three steps.

{\it Step 1.}
Define the subset of $X\times  \erre$
\[
\hat A:= \big\{ (x,t) : x \in A_{t} \big\}. 
\]
Note that 
\begin{eqnarray*}
\hat A & = &~ P_{13} \bigg\{ (x,y,t) \in X \times X \times \erre^{+} : (x,y) \in G \cap A \times X, d_{N}(x,y)=t \bigg\}  \crcr
&&~ \qquad  \cup P_{13} \bigg\{ (x,y,t) \in X \times X \times \erre^{-} : (x,y) \in G^{-1} \cap A \times X, d_{N}(x,y)= -t \bigg\},
\end{eqnarray*}
hence it is analytic. Clearly $A_{t} = \hat A (t)$. 

{\it Step 2.} Define the closed set in $\mathcal{P}(X \times [0,1])$
$$
\Pi(\mu) : = \big\{ \pi \in \mathcal{P}(X\times [0,1]) : (P_{1})_{\sharp}(\pi)= \mu \big\}
$$
and let $B \subset X\times \R \times [0,1]$ be a Borel set such that $P_{12}(B)= \hat A$.

Consider the function 
\[ 
\R \times \Pi(\mu) \ni (t, \pi) \mapsto \pi(B(t)).
\]
A slight modification of Lemma 4.12 in \cite{biacar:cmono} shows that this function is Borel.

{\it Step 3.}
Since supremum of Borel function are $\mathcal{A}$-measurable, pag. 134 of \cite{Sri:courseborel},  
the proof is concluded once we show that 
\[
\mu(A_{t}) = \mu (\hat A(t)) = \sup_{ \pi \in \Pi(\mu) } \pi(B(t)). 
\]

Since $\hat A(t) \times [0,1] \supset B(t)$
\[ 
\mu (\hat A(t)) =   \pi(\hat A(t) \times [0,1]) \geq \pi (B(t)).
\]
On the other hand from Theorem \ref{T:vanneuma}, there exists an $\mathcal{A}$-measurable section of the analytic set $B(t)$, 
so we have $u: \hat A(t) \to B(t)$.
Clearly for $\pi_{u} = (\Id,u)_{\sharp}(\mu)$ it holds $\pi_{u}(B(t))= \mu (\hat A(t))$.

\end{proof}

The next assumption is the fundamental assumption of the paper.

\begin{assumption}[Non-degeneracy assumption]
\label{A:NDEatom} The measure $\mu$ satisfies Assumption \ref{A:NDEatom}
if for each analytic set $A\subset \mathcal{T}_{e}$ there exists a sequence $\{ t_{n} \}_{n\in \enne} \subset \erre$ and a strictly positive constant $C$ such that $t_{n} \to 0$ as $n\to +\infty$ and $\mu(A_{t_{n}}) \geq C \mu (A)$ for every $n \in \enne$.
\end{assumption}
Note that Assumption \ref{A:NDEatom} was already introduced at page \pageref{I:NDEatom}.
Clearly it is enough to verify Assumption \ref{A:NDEatom} for $A$ compact set. 
An immediate consequence of the Assumption \ref{A:NDEatom} is that the measure $\mu$ is concentrated on $\mathcal{T}$.

\begin{proposition}
\label{P:puntini}
If $\mu$ satisfies Assumption \ref{A:NDEatom} then
\[
\mu(\mathcal{T}_e\setminus \mathcal{T}) = 0.
\]
\end{proposition}

\begin{proof}
Let $A= \mathcal{T}_{e}\setminus \mathcal{T}$. Suppose by contradiction $\mu(A)>0$.
By the inner regularity there exists $\hat A \subset A$ closed with $\mu (\hat A)>0$. By Assumption \ref{A:NDEatom} there exist $C > 0$ and 
$\{ t_{n} \}_{n\in \enne}$ converging to 0 such that $\mu(\hat A_{t_{n}}) \geq C \mu(\hat A)$. 
 
Define $ \hat A^{\ve}:= \big\{ x : d_N(\hat A,x) < \ve \big\}$. 
Since $\hat A \subset A$, for all $n \in \enne$ it holds $\hat A_{t_{n}} \cap A = \emptyset$. 
Moreover for $t_{n}\leq\ve$ we have $\hat A^{\ve} \supset \hat A_{t_{n}}$. So we have
\[ 
\mu(\hat A) = \lim_{\ve \to 0} \mu (\hat A^{\ve}) \geq \mu(\hat A) + \mu(\hat A_{t_{n}}) \geq (1+ C) \mu (\hat A),
\]
that gives the contradiction.
\end{proof}

Once we know that $\mu(\mathcal{T}) = 1$, we can use the Disintegration Theorem \ref{T:disintr} to write
\begin{equation}
\label{E:disintT}
\mu = \int_S \mu_y m(dy), \quad m = f_\sharp \mu, \ \mu_y \in \mathcal{P}(H(y)).
\end{equation}
The disintegration is strongly consistent since the quotient map $f : \mathcal T \to \mathcal T$ is $\mu$-measurable and $(\mathcal T,\mathcal B(\mathcal T))$ is countably generated.

The second consequence of Assumption \ref{A:NDEatom} is that $\mu_y$ is continuous, i.e. $\mu_y(\{x\}) = 0$ for all $x \in X$.

\begin{proposition}
\label{P:nonatoms}
If $\mu$ satisfies Assumption \ref{A:NDEatom} then the conditional probabilities $\mu_y$ are continuous for $m$-a.e. $y \in S$.
\end{proposition}

\begin{proof}
From the regularity of the disintegration and the fact that $m(S) = 1$, we can assume that the map $y \mapsto \mu_y$ is weakly continuous on a compact set $K \subset S$ of comeasure $<\ve$. 
It is enough to prove the proposition on $K$.

{\it Step 1.} From the continuity of $K \ni y \mapsto \mu_y \in \mathcal{P}(X)$ w.r.t. the weak topology, it follows that the map
\[
y \mapsto A(y) := \big\{ x \in H(y): \mu_y(\{x\}) > 0 \big\} = \cup_n \big\{ x \in H(y): \mu_y(\{x\}) \geq 2^{-n} \big\}
\]
is $\sigma$-closed: in fact, if $(y_m,x_m) \to (y,x)$ and $\mu_{y_m}(\{x_m\}) \geq 2^{-n}$, then $\mu_y(\{x\}) \geq 2^{-n}$ by u.s.c. on compact sets.
Hence $A$ is Borel, where $A = \{ (y,A(y)) : y\in K \}$. 

{\it Step 2.} The claim is equivalent to $\mu(P_{2}(A))=0$. Suppose by contradiction $\mu(P_{2}(A))>0$.
By Lusin Theorem (Theorem 5.8.11 of \cite{Sri:courseborel}) 
$A$ is the countable union of Borel graphs, $A = \cup_{n} A_{n}$.
Therefore we can take a Borel selection of $A$ just considering one of the Borel graphs, say $\hat A$.
Since at least one of $P_{2}(A_{n})$ must have positive $\mu$-measure, we can assume $\mu (P_{2}(\hat A))>0$.
 

By Assumption \ref{A:NDEatom} 
$\mu( (P_{2}(\hat A))_{t_{n}} ) \geq C \mu(P_{2}(\hat A))$ for some $C>0$ and $t_{n}\to 0$.
Since $\hat A$ is a Borel graph, for every $y \in P_{1}(\hat A)$ the set $P_{2}( \{y\} \times X \cap \hat A)$ is a singleton.
Hence $(P_{2}(\hat A))_{t_{n}} \cap (P_{2}(\hat A))_{t_{m}} = \emptyset$. We have a contradiction with the fact that the measure is finite.
\end{proof}

\section{Solution to the Monge problem}
\label{S:Solution}
Throughout the section we assume $\mu$ to satisfy Assumption \ref{A:NDEatom}. 
It follows from Disintegration Theorem \ref{T:disintr}, Proposition \ref{P:puntini} and Proposition \ref{P:nonatoms} that
\[
\mu = \int \mu_y m(dy), \ \pi = \int \pi_y m(dy), \quad \mu_y \ \textrm{continuous}, \ (P_1)_\sharp \pi_y = \mu_y,
\]
where $m = f_\sharp \mu$ and $\mu_y \in \mathcal{P}(H(y))$. 
We write moreover
\[
\nu = \int \nu_y m(dy) = \int (P_2)_\sharp \pi_y m(dy).
\]
Note that $\pi_y \in \Pi(\mu_y,\nu_y)$ is $d_N$-cyclically monotone and (since $d_{N}\llcorner_{H(y) \times H(y)}$  is finite and,
from point \ref{A:comegeo} of Assumption \ref{A:assu1}, lower semi-continuous) optimal for $m$-a.e. $y$. 
If $\nu(\mathcal{T}) = 1$, then the above formula is the disintegration of $\nu$ w.r.t. $H$.

\begin{theorem}\label{T:finale}
Assume that for all $y \in S$ there exists an optimal map $T_{y}$ from $\mu_{y}$ to $\nu_{y}$.
Then there exists a $\mu$-measurable map $T: X \to X$ such that 
\[
\int d_{N}(x,T(x)) \mu(dx) = \int d_{N}(x,z) \pi(dxdz), \qquad T_{\sharp} \mu = \nu.
\] 
\end{theorem}
Recall $S \subset \mathcal{T}$ introduced in (\ref{E:TNngel}).

\begin{proof}
The idea is to use Theorem \ref{T:vanneuma}.

{\it Step 1.} Let $\mathbf{T} \subset S \times \mathcal{P}(X^2)$ be the set: for $y \in S$, $\mathbf{T}_y$ is the family of optimal transference plans in $\Pi(\tilde \mu_y,\tilde \nu_y)$ concentrated on a graph,
\[
\mathbf{T} = \Big\{ (y, \pi) \in S \times \mathcal{P}(X^2): \pi \in \Pi(\mu_y, \nu_y) \ \textrm{optimal}, \exists T : X \to X, \pi(\textrm{graph}(T)) = 1 \Big\}.
\]
where for optimal in $\Pi(\mu_{y},\nu_{y})$ we mean 
\[ 
\int d_{N} \pi = \min_{\pi \in \Pi(\mu_{y},\nu_{y})}\int d_{N} \pi. 
\]
Note that, since $\pi$ is a Borel measure, in the definition of $\mathbf{T}$, $T$ can be taken Borel.
Moreover the $y$ section $\mathbf{T}_y = \mathbf{T} \cap \{y\} \times \mathcal{P}(X^2)$ is not empty.

{\it Step 2.} Since the projection is a continuous map, then the set
\[
\tilde \Pi = \Big\{ (y, \pi): (P_1)_\sharp \pi =  \mu_y, (P_2)_\sharp \pi =  \nu_y \Big\}
\]
is a Borel subsets of $S \times \mathcal{P}(X^2)$: in fact it is the counter-image of the Borel set $\textrm{graph}( (\mu_y, \nu_y)) \subset 
S \times \mathcal{P}(X)^2$ w.r.t. the weakly continuous map $(y,\pi) \mapsto (y,(P_1)_\sharp \pi,(P_2)_\sharp  \pi)$.

Define the Borel function
\[
S \times \mathcal{P}(X^2) \ni (y, \pi) \mapsto f(y, \pi) :=
\begin{cases}
\int d_{N}  \pi &  \pi \in \Pi( \mu_y, \nu_y) \crcr
+ \infty & \textrm{otherwise}
\end{cases}
\]
It follows that $y \mapsto g(y) := \inf_{\pi} f(y, \pi)$ is an $\mathcal{A}$-function: we can redefine it on a $m$-negligible set to make it Borel, 
where $m$ is the quotient measure of $\mu$. Hence the set
\[
\tilde \Pi^{\textrm{opt}} = \bigg\{ (y, \pi):  \pi \in \tilde \Pi( \mu_y, \nu_y), \int d_{N} \pi \leq g(y) \bigg\} = \tilde \Pi \cap \Big\{ (y, \pi): \int d_{N}  \pi \leq g(y) \Big\}
\]
is Borel.

{\it Step 3.}  Now we show that the set of $\pi \in \mathcal{P}(X^{2})$ concentrated on a graph is analytic. By Borel Isomorphism Theorem, 
see \cite{Sri:courseborel} page 99, it is enough to prove the same statement for $\pi \in \mathcal{P}([0,1])^{2}$.
Consider the function
\[
 \mathcal{P}([0,1]^{2}) \times C_b([0,1],[0,1]) \ni ( \pi,\phi) \mapsto h( \pi,\phi) :=  \pi(\textrm{graph}(\phi)) \in [0,1].
\]
Since $\textrm{graph}(\phi)$ is compact, $h$ is u.s.c.. Hence the set $B^n = h^{-1}([1-2^{-n},1])$ is closed, so that
\[
\mathscr{T} = \bigcap_n P_1(B^n) = \Big\{  \pi: \forall \varepsilon > 0\ \exists \phi_\varepsilon,  \pi(\phi_\varepsilon) > 1 - \varepsilon \Big\}
\]
is an analytic set. It is easy to prove that $\pi \in \mathscr{T}$ iff $\pi$ is concentrated on a graph.

{\it Step 4.} It follows that
\[
\mathbf{T} = S \times \mathscr{T} \cap \tilde \Pi^{\textrm{opt}}
\]
is analytic and by Theorem \ref{T:vanneuma} there exists a $m$-measurable selection $y \mapsto \pi_y \in \mathbf{T}_y$. It is fairly easy to prove that $\int \pi_y m(dy)$ is concentrated on a graph, has the same transference cost of $\pi$ and belongs to $\Pi(\mu,\nu)$.
\end{proof}

It follows from Theorem \ref{T:finale}
that it is enough to solve for each $y \in S$ the Monge minimization problem with marginal $\mu_{y}$ and $\nu_{y}$ on the set $H(y)$. 
In order to solve it, we introduce an assumption on the geometry of the set $H(y)$.

\begin{assumption}
\label{A:clessidra3}
For a given $y \in S$, $H(y)$ satisfies Assumption \ref{A:clessidra3} if  
there exist two families of disjoint $\mathcal{A}$-measurable sets $\{ K_{t}\}_{t \in [0,1]}$ and $\{Q_{s} \}_{s\in [0,1]}$ such that 
\begin{itemize}
\item $\mu_{y}( H(y) \setminus \cup_{t \in [0,1]} K_{t} ) =  \nu_{y}( H(y) \setminus \cup_{s \in [0,1]} Q_{s} )=0$;
\item the associated quotient maps $\f_{K}$ and $\f_{Q}$ are respectively $\mu_{y}$-measurable and $\nu_{y}$-measurable; 
\item for $t\leq s$, $K_{t} \times Q_{s}\subset G$.
\end{itemize}
\end{assumption}
Note that Assumption \ref{A:clessidra3} was already introduced at page \pageref{I:clessidra3}.
In the measurability condition of Assumption \ref{A:clessidra3}, the set $[0,1]$ is equipped with the Borel $\sigma$-algebra $\mathcal{B}([0,1])$.
If $H(y)$ satisfies Assumption \ref{A:clessidra3} we can disintegrate the marginal measures $\mu_{y}$ and $\nu_{y}$ respectively
w.r.t. the family $\{K_{t}\}$ and $\{ Q_{s} \}$: 
\[
\mu_{y} = \int \mu_{y,t} m_{\mu_{y}}(dt), \quad \nu_{y} = \int \nu_{y,t} m_{\nu_{y}}(dt)
\] 
where $m_{\mu_{y}} = \f_{K \, \sharp} \mu_{y}$, $m_{\nu_{y}} = \f_{Q \, \sharp} \nu_{y}$ and the disintegrations are strongly consistent.

\begin{proposition}\label{P:alternativa}
Suppose that $H(y)$ satisfies Assumption \ref{A:clessidra3} and that the following conditions hold true:
\begin{itemize}
\item $m_{\mu_{y}}$ is continuous;
\item $\mu_{y,t}$ is continuous for $m_{\mu_{y}}$-a.e. $t \in [0,1]$;
\item $m_{\mu_{y}} ([0,t]) \geq m_{\nu_{y}} ([0,t])$ for $m_{\mu_{y}}$-a.e. $t \in [0,1]$.
\end{itemize}
Then there exists a $d_{N}$-cyclically monotone $\mu_{y}$-measurable map $T_{y}$ such that $T_{y\, \sharp} \mu_{y} = \nu_{y}$ and 
\[
\int d_{N}(x,T_{y}(x)) \mu_{y}(dx) = \int d_{N}(x,z) \pi_{y}(dxdz). 
\]
\end{proposition}

\begin{proof}
{\it Step 1.}
Since $m_{\mu_{y}}$ is continuous and $m_{\mu_{y}} ([0,t]) \geq m_{\nu_{y}} ([0,t])$,
there exists an increasing map $\psi : [0,1] \to [0,1]$ such that  $\psi_{\sharp} m_{\mu_{y}} = m_{\nu_{y}}$. 

Moreover, since for $m_{\mu_{y}}$-a.e. $t \in [0,1]$ $\mu_{y,t}$ is continuous,  
there exists a Borel map $T_{t} : K_{t} \to Q_{\psi(t)}$ such that 
$T_{t\,\sharp}\mu_{y,t} = \nu_{y,\psi(t)}$ for $m_{\mu_{y}}$-a.e. $t \in [0,1]$. 
Since $\psi(t) \geq t$ the map $T_{t}$ is $d_{N}$-cyclically monotone, hence optimal between $\mu_{y,t}$ and $\nu_{y,t}$.

{\it Step 2.} Reasoning as in the proof of Theorem \ref{T:finale}, one can prove the existence of a $\mu_{y}$-measurable map $T_{y}: H(y) \to H(y)$ 
that is the gluing of all the maps $T_{t}$ constructed in {\it Step 1.}.
Hence there exists a $\mu_{y}$-measurable map $T_{y}: H(y) \to H(y)$ such that $T_{y\,\sharp} \mu_{y,t} = \nu_{y,\psi(t)}$. 
It follows from Assumption \ref{A:clessidra3} that $\gr( T_{y}) \subset G$, hence $T_{y}$ is $d_{N}$-cyclically monotone and  
\[
T_{\sharp} \mu_{y} = \int T_{\sharp}\mu_{y,t} m_{\mu_{y}}(dt) = \int \nu_{y,\psi(t)} m_{\mu_{y}}(dt) = 
\int \nu_{y,t} (\psi_{\sharp}m_{\mu_{y}})(dt) = \nu_{y}.  
\]
\end{proof}

The next corollary follows straightforwardly and it sums up all the results.

\begin{corollary}
Let $\pi \in \Pi(\mu,\nu)$ be concentrated on a $d_{N}$-cyclically monotone set $\Gamma$ satisfying Assumption \ref{A:assu1}.
Assume that $\mu$ satisfies Assumption \ref{A:NDEatom} and for $m$-a.e. $y \in S$ the set $H(y)$ satisfies Assumption \ref{A:clessidra3}. 
If for $m$-a.e. $y \in S$ the hypothesis of Proposition \ref{P:alternativa} are verified,
then there exists an Borel map $T : X \to X$ such that
\[
\int d_{N}(x,T(x)) \mu(dx) = \int d_{N}(x,z) \pi(dxdz), \qquad T_{\sharp} \mu = \nu.
\]
If $\pi$ is also optimal, then $T$ solves the Monge minimization problem.
\end{corollary}

Let us summarize the theoretical results obtained so far.
Let $\pi \in \Pi(\mu,\nu)$ be a $d_{N}$-cyclically monotone transference plan concentrated on a set $\Gamma$. 
Consider the corresponding family of chain of transport rays and assume  
that $\Gamma$ satisfies Assumption \ref{A:assu1}. 
Then the partition induced by $H$ permits to obtain a strongly consistent disintegration formula of $\mu,\nu$ and $\pi$ holds.
If $\mu$ satisfies Assumption \ref{A:NDEatom} then the set of initial points is $\mu$-negligible and 
the conditional probabilities $\mu_{y}$ are continuous.
 
Since the geometry of $H(y)$ can be wild, we need another assumption to build a $d_{N}$-monotone transference map between $\mu_{y}$ 
and $\nu_{y}$. 
If $H(y)$ satisfies Assumption \ref{A:clessidra3} we can perform another disintegration and, under additional regularity of 
the conditional probabilities of $\mu_{y}$ and of the quotient measure of $\mu_{y}$, we prove the existence of a 
$d_{N}$-monotone transference map between $\mu_{y}$ and $\nu_{y}$. 
Applying the same reasoning for $m$-a.e. $y$ we prove the existence of a transport map $T$ between $\mu$ and $\nu$ that has the 
same transference cost of the given $d_{N}$-cyclically monotone plan $\pi$.

\subsection{Example}\label{ss:esempio}

We conclude this Section with the analysis of a particular case in which the set $H(y)$ satisfies Assumption \ref{A:clessidra3}.  
The hypothesis of Proposition \ref{P:alternativa} and Assumption \ref{A:clessidra3} 
were partially inspired by this example. 
What follows will be useful in the next Section, 
however, since it is not only related to what will be proved in Section \ref{S:application}, we have decided to present it here.

Fix the following notation: a continuous curve $\gamma:[0,1] \to X$ is \emph{increasing} if for $t,s \in [0,1]$
\[
t\leq s \Longrightarrow (\gamma(t),\gamma(s)) \in G
\]
\begin{definition}[Hourglass sets]\label{D:clessidra}
For $z \in X$ define the \emph{hourglass set}
\[
K(z):= \Big\{ (x,y) \in X \times X :  (x,z), (z,y) \in G   \Big\}.
\]
\end{definition}
Assume that there exists an increasing curve $\gamma$ such that  
\[ 
H(y) \times X \cap \Gamma \subset \bigcup_{t \in [0,1]}  K(\gamma(t)) \cap \Gamma.
\]
Note that this assumption is equivalent to request 
that on each chain of transport rays the branching structures can appear only along an increasing curve $\gamma$.

\begin{figure}
\label{Fi:clessidra}
\psfrag{g}{$\gamma$}
\psfrag{x}{$x$}
\psfrag{y}{$y$}
\psfrag{Z}{$z$}
\centerline{\resizebox{11cm}{3cm}{\includegraphics{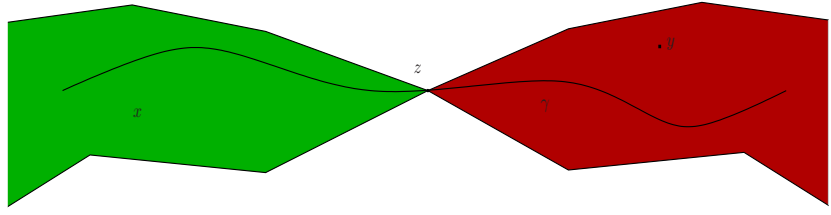}}}
\caption{The hourglass set $K(z)$.}
\end{figure}

Then $H(y)$ satisfies Assumption \ref{A:clessidra3}. 
Indeed first notice that $K(z)$ is analytic, then define the family of sets
\[
K_{t} : = G^{-1}(\gamma(t)) \setminus \bigcap_{s < t} G^{-1}(\gamma(s)), \qquad 
Q_{t} : = G(\gamma(t)) \setminus \bigcap_{t < s} G(\gamma(s)).
\]
Since $\gamma$ is increasing,  $K_{t}$ and $Q_{s}$ are $\mathcal{A}$-measurable and the quotient maps are $\mathcal{A}$-measurable: 
let $[a,b] \subset [0,1]$
\[
\f_{K}^{-1}([a,b]) = \bigcup_{t \in [a,b]} K_{t} = G^{-1}(\gamma(b)) \setminus G^{-1}(\gamma(a)) \cup K_{a}\, \in \mathcal{A}
\]
and the same calculation holds true for $\f_{Q}$. From the increasing property of $\gamma$ it follows that $K_{t} \times Q_{s} \in G$ for every 
$0 \leq t \leq s \leq 1$.
Again from the increasing property of $\gamma$ it follows that $m_{\mu_{y}}([0,t]) \geq m_{\nu_{y}}([0,t])$.

\bigskip

\section{An application}\label{S:application}

Throughout this section $|\cdot|$ will be the euclidean distance of $\erre^{d}$.

Let $C \subset \erre^{d}$ be an open convex set such that $M:=\partial C$ is a smooth compact sub-manifold of $\erre^{d}$ of dimension $d-1$. 
Let $X: = \erre^{d}\setminus C$. Clearly $X$ endowed with the euclidean topology is a Polish space.

Consider the following geodesic distance: $d_{M}:X\times X \to [0,+\infty]$:  
\begin{equation}\label{E:distostc}
d_{M}(x,y) : = \inf \{ L(\gamma): \gamma \in \lip ([0,1], X), \gamma(0)=x,\gamma(1)=y \},
\end{equation}
where $L$ is the standard euclidean arc-length: $L(\gamma)=\int |\dot \gamma|$. 
Hence $M$ can be seen as an obstacle for geodesics connecting points in $X$.
Note that any minimizing sequence has uniformly bounded Lipschitz constant, therefore in the definition of $d_{M}$ we can substitute $\inf$ with $\min$. 
Hence $d_{M}$ is a geodesic distance on $X$.

We will show that given $\mu,\nu \in \mathcal{P}(\erre^{d})$ with $\mu\ll \mathcal{L}^{d}$, 
the Monge minimization problem with geodesic cost $d_{M}$ admits a solution.

From now on we will assume that $\mu \ll \mathcal{L}^{d}$.
and all the sets and structures introduced during the paper will be referred to this Monge problem.

The strategy to solve the Monge minimization problem is the one used in Section \ref{S:Solution}: 
build an optimal map on each equivalence class $H(y)$ and then use Theorem \ref{T:finale}. 
To prove the existence these optimal maps we will show that the geometry of the chain of transport rays $H(y)$
is the one presented in Example \ref{ss:esempio} and that the hypothesis of Proposition \ref{P:alternativa} are satisfied.

\begin{lemma}\label{L:cont}
The distance $d_{M}$ is a continuous map.
\end{lemma}
\begin{proof}
{\it Step 1.} Let $\{x_{n}\}_{n\in \enne},\{y_{n}\}_{n\in \enne} \in X$ such that $|x_{n} - x| \to 0$  $|y_{n} - y| \to 0$. 
Since the boundary of $X$ is a smooth manifold,
for every $n \in \enne$ there exist curves $\gamma_{1,n},\gamma_{2,n} \in \lip([0,1],X)$ 
such that 
\begin{itemize}
\item $\gamma_{1,n}(0)=x, \gamma_{1,n}(1)=x_{n}$; 
\item $\gamma_{2,n}(0)=y, \gamma_{2,n}(1)=y_{n}$; 
\item $L(\gamma_{i,n}) \to 0$ as $n \to +\infty$, for $i=1,2$.
\end{itemize}
Consider $\gamma_{n}\in \lip([0,1],X)$ such that $\gamma_{n}(0)=x_{n}, \gamma_{n}(1)=y_{n}$ and $L(\gamma_{n})\leq d_{M}(x_{n},y_{n}) + 2^{-n}$. 
Gluing $\gamma_{1,n}$ and $\gamma_{2,n}$ to $\gamma_{n}$ it follows
\[
d_{M}(x,y) \leq d_{M}(x_{n},y_{n}) + 2^{-n} + L(\gamma_{1,n}) + L(\gamma_{2,n}).
\]
Hence $d_{M}$ is l.s.c..

{\it Step 2.} Taking a minimizing sequence of admissible curves for $d_{M}(x,y)$ and gluing them with $\gamma_{i,n}$ as in \emph{Step 1.}, 
it is fairly easy to prove that $d_{M}$ is u.s.c. and therefore continuous.
\end{proof}

As a corollary we have the existence of an optimal transference plan $\pi$. Hence from now on $\pi$ will be an optimal transference plan
and all the structures defined during the paper starting from a generic $d_{N}$-cyclically monotone plan, are referred to it.
Moreover there exists $\varphi \in \lip_{d_{M}}(X,\erre)$ such that $\Gamma =\Gamma' = G = \{(x,y)\in X\times X: \varphi(x)-\varphi(y)=d_{M}(x,y)\}$. 
Note that $\Gamma$ is closed.

The next result shows that  the sets $H(y)$ have the structure of Example \ref{ss:esempio}. 
The convex assumption on the obstacle is fundamental: 
each transport rays is composed by a straight line, a geodesic on $M$ where branching structures are allowed and again a straight line.

\begin{lemma}\label{L:geoH}
For all $y \in S$, $H(y)$ has the geometry of Example \ref{ss:esempio}: 
there exists an increasing curve $\gamma_{y}: [0,1] \to X$ such that
\[ 
H(y) \times X \cap \Gamma \subset \bigcup_{t \in [0,1]}  K(\gamma_{y}(t)) \cap \Gamma.
\]
\end{lemma}
\begin{proof}
Since due to convexity and smoothness of the obstacle, the geodesics of $d_{M}$ are smooth and composed by a first straight line, a geodesic of the manifold and a final straight line, a branching structure can appear only on the manifold $M$. 
If $H(y) \neq R(y)$, consider the following sets: 
\[
Z:= \bigcap_{z\in  H(y)\cap M}  G^{-1}(z) \cap M, \quad W:= \bigcap_{z \in  H(y)\cap M}  G(z) \cap M.
\]
By $d_{M}$-monotonicity, smoothness and convexity of $M$, for all $z\in H(y)\cap M$ the set $ G^{-1}(z) \cap M$ is 
always contained in the same geodesic of $M$. 
Using the compactness of $M$, $Z= \{ z\}$ and $W = \{w \}$ and $(z,w) \in G$. 
Consider the unique increasing geodesic $\gamma_{y} \in \gamma_{[z,w]}$ such that $\gamma_{y} = G(z) \cap G^{-1}(w)$.
Hence 
\[ 
H(y) \times X \cap \Gamma \subset \bigcup_{t \in [0,1]}  K(\gamma_{y}(t)) \cap \Gamma.
\]
\end{proof}

\begin{remark}\label{O:sempre}
From Lemma \ref{L:cont} and Lemma \ref{L:geoH} it follows that, for any transference plan $\pi$, 
the set of chain of transport rays $H$ satisfies Assumption \ref{A:assu1}. 

Indeed consider $H(y)$ and the corresponding geodesic $\gamma_{y}$ from Lemma \ref{L:geoH}.
Then take any sequence $x_{n} \in H(y)$ such that $|x_{n} - x| \to 0$ as $n \to +\infty$. 
Note that there exist $s_{n} \in [0,1]$ and $t_{n}\in \erre$ such that 
$x_{n} =  \gamma_{y}(s_{n}) + t_{n} \nabla\gamma_{y}(s_{n})$. Possibly passing to subsequences, $s_{n}\to s$, $t_{n} \to t$ with 
$x = \gamma_{y}(s) + t \nabla\gamma_{y}(s)$.
Since $(\gamma_{y}(s_{n}), x_{n}) \in G$ and $G$ is closed it follows that $(\gamma_{y}(s), x) \in G$.
From $(y,\gamma_{y}(s)) \in R$ follows $x \in H(y)$. Hence point \ref{A:comegeo} of Assumption \ref{A:assu1} holds true.

Point \ref{A:comelocpt} of Assumption \ref{A:assu1} follows directly from the continuity of $d_{N}$.
\end{remark}

In the following Lemma we prove that the problem can be reduced to the equivalence classes $H(y)$.
We use the following notation: the quotient map induced by $H$ will be denoted by $f^{y}$ and the corresponding quotient measure $f^{y}_{\sharp}\mu$ by $m_{H}$.

\begin{lemma}\label{L:appl}
The $\mu$-measure of the set of initial points is zero, hence
\[
\mu = \int \mu_{y} m_{H}(dy).
\]
Moreover $\mu_{y}$ is continuous for $m_{H}$-a.e. $y$. 
\end{lemma}
\begin{proof}

{\it Step 1.} Since $\mu\ll \mathcal{L}^{d}$, it is enough to prove that the set of initial points is $\mathcal{L}^{d}$-negligible and 
that the disintegration w.r.t. $H$ of $\mathcal{L}^{d}$ restricted to any compact set has continuous conditional probabilities. 
Indeed if $\mathcal{L}^{d}\llcorner_{K} = \int \eta_{y} m_{\mathcal{L}^{d}}(dy)$ and $\mu = \rho \mathcal{L}^{d}$ then $m_{\mu}\ll m_{\mathcal{L}^{d}}$ and
\[
\mu \llcorner_{K} = \int \rho \eta_{y} m_{\mathcal{L}^{d}}(dy) =  \int \rho \frac{d m_{\mathcal{L}^{d}}}{d m_{\mu}} \eta_{y} m_{\mu}(dy),
\]
where $m_{\mu}$ is the quotient measure of $\mu\llcorner_{K}$. 
It follows that the continuity of $\eta_{y}$ implies the continuity of conditional probabilities of $\mu$.
Hence the claim is to prove that $\mathcal{L}^{d}$ satisfies Assumption \ref{A:NDEatom}.

{\it Step 2.} 
Let $K \subset X$ be any compact set with $\mathcal{L}^{d}(K)>0$. 
Possibly intersecting $K$ with  $B_{r}(x)$ for some $x \in \erre^{d} \setminus C$ and $r>0$, 
we can assume w.l.o.g. that $K\subset B_{\ve}(x)$ and $B_{2\ve}(x) \cap M = \emptyset$. 
Since $d_{M}\geq d$, $K_{t}\subset B_{2\ve}(x)$ for all $t\leq\ve$. 
Since $d_{M}= |\cdot|$ in $B_{2\ve}(x)$, 
it follows that inside $B_{2\ve}(x) \times B_{2\ve}(x)$ $d_{M}$-cyclically monotonicity is 
equivalent to $|\cdot|$-cyclically monotonicity. It follows that the set $H \cap G(K) \times G^{-1}(K_{\ve})$ is $|\cdot|$-cyclically monotone.

{\it Step 3.} The following is proved in \cite{biacava:streconv}: consider a metric measure space $(X,d,m)$ with $d$ non-branching geodesic distance, 
 $m \in \mathcal{P}(X)$ and assume that $(X,d,m)$ satisfies $MCP(K,N)$, for the definition of $MCP(K,N)$ we refer to \cite{Sturm:MGH2}. Let $\Gamma$ be a $d$-cyclically monotone set and consider the evolution of sets induced by $\Gamma$, then $m$ satisfies Assumption \ref{A:NDEatom} w.r.t. this evolution of sets. 
 
Since $B_{2\ve}(x)$ is a convex set, it follows that $(B_{2\ve}(x),|\cdot|, \mathcal{L}^{d})$ satisfies $MCP(0,d)$. Therefore 
$\mathcal{L}^{d}$ satisfies Assumption \ref{A:NDEatom} w.r.t. the evolution of sets induced by $H \cap G(K) \times G^{-1}(K_{\ve})$.
The claim follows. 
\end{proof}

Hence we can assume w.l.o.g.  that $\mu(G^{-1}(M)) = \nu (G (M)) = 1$:
if $H(y)$ do not intersect the obstacle, it is a straight line and the marginal $\mu_{y}$ 
is continuous. Since the existence of an optimal transport map on a straight line with first marginal continuous is a standard fact in optimal transportation, the reduction follows.


Recall the two family of sets introduced in Example \ref{ss:esempio}:
\[
K_{y,t} : = G^{-1}(\gamma_{y}(t)) \setminus \bigcap_{s < t} G^{-1}(\gamma_{y}(s)), \qquad 
Q_{y,t} : = G(\gamma_{y}(t)) \setminus \bigcap_{t < s} G(\gamma_{y}(s)).
\]
It follows from Lemma \ref{L:geoH} and Example \ref{ss:esempio} that 
\[
\mu_{y} = \int \mu_{y,t} m_{\mu_{y}}(dt), \quad \nu_{y} = \int \nu_{y,t} m_{\nu_{y}}(dt).
\]
with $\mu_{y,t}(K_{y,t}) =  \nu_{y,t}(Q_{y,t}) =1$. 
Moreover using the increasing curve $\gamma_{y}$, we can assume that  $m_{\mu_{y}} \in \mathcal{P}(M)$, indeed
\begin{equation}\label{E:primaM}
\mu_{y} = \int_{[0,1]} \mu_{y,t} m_{\mu_{y}}(dt) = \int_{\gamma_{y}([0,1])} \mu_{y, \gamma^{-1}_{y}(z)} (\gamma_{y \, \sharp} m_{\mu_{y}})(dz).
\end{equation}
And the same calculation holds true for $\nu_{y}$ and $m_{\nu_{y} } $.
Therefore in the following 
\begin{equation}\label{E:primaM2}
\mu_{y} = \int_{M} \mu_{y,z} m_{\mu_{y}}(dz), \quad \nu_{y} = \int_{M} \nu_{y,z} m_{\nu_{y}}(dz)
\end{equation}
with $\mu_{y,z}( K_{y, \gamma_{y}^{-1}(z)} ) = \nu_{y,z}( Q_{y, \gamma_{y}^{-1}(z)} ) = 1$ and $m_{\mu_{y}}(\gamma_{y}([0,1])) =  m_{\nu_{y}}(\gamma_{y}([0,1]))=1$.

Moreover w.l.o.g. we can assume that $S = f^{y}(\erre^{d}) \subset M$, in particular we can assume that for all  $y \in S$ there exists $t(y) \in [0,1]$ 
such that $y = \gamma_{y}(t(y))$.

According to Proposition \ref{P:alternativa},
to obtain the existence of an optimal map on $H(y)$ it is enough to prove that
$m_{\mu_{y}}$ is continuous and $\mu_{y,z}$ is continuous for $m_{\mu_{y}}$-a.e. $z \in M$.
Recall that $m_{\mu_{y}}( \gamma_{y}([0,t])) \geq m_{\nu_{y}} (\gamma_{y}([0,t]))$ is a straightforward consequence of 
the increasing property of $\gamma_{y}$.

\begin{remark}\label{R:tuttoM}
Consider the following $\mathcal{A}$-measurable map: 
\[
G^{-1}(M) \setminus (a(M) \cap M) \ni w \mapsto  f^{M}(w): = \textrm{Argmin} \{ d(z,w) : z \in M \cap G(w) \} \in M.
\]
Consider the measure $m : = f^{M}_{\sharp} \mu \in \mathcal{P}(M)$. 
Observing that $f^{M}(H(y)) = \gamma_{y}([0,1])$, it follows that 
the support of $m$ is partitioned by a $d_{M}$-cyclically monotone equivalence relation: 
\[
m \Big(  \bigcup_{y \in S}  \gamma_{y}([0,1]) \Big)=1, \qquad  \bigcup_{y \in S}  \gamma_{y}([0,1]) \times \bigcup_{y \in S}  \gamma_{y}([0,1]) \cap G \ \ \textrm{is $d_{M}$-cyclically monotone}
\]
Moreover $f^{y}$ is a quotient map also for this equivalence relation. 
Note that $f^{y}_{\sharp} m = m_{H}$: consider $I \subset S$

\begin{eqnarray*} 
(f^{y}_{\sharp} m )(I) = &~ m \Big(\bigcup_{y \in I}  \gamma_{y}([0,1])  \Big)  =  \mu \Big(  G^{-1}(\bigcup_{y \in I}  \gamma_{y} ([0,1]) ) \Big) \crcr
=&~  \mu \Big( \bigcup_{y \in I} H(y) \Big) = (f^{y}_{\sharp}\mu)(I) = m_{H}(I).
\end{eqnarray*}
It follows that  
\[
m = \int_{S} ( f^{M}_{\sharp} \mu_{y} ) m_{H}(dy)
\]
and from (\ref{E:primaM2})  $f^{M}_{\sharp} \mu_{y} = m_{\mu_{y}}$. 
Hence the final disintegration formula for $m$ is the following one: 
\begin{equation}\label{E:dopoM}
m = \int_{S}  m_{\mu_{y}}  m_{H}(dy).
\end{equation}
\end{remark}

\begin{proposition}\label{P:fondamentale}
The measure $m$ is absolutely continuous w.r.t. the Hausdorff measure $\haus^{d-1}$ restricted to M.
\end{proposition}
\begin{proof}
Recall that $\f \in \lip_{d_{M}}(\erre^{d})$ is the potential associated to $\Gamma$ and consider the following set
\[
M_{2}: = P_{1} \Big(   \{(x,y) \in M\times M : | \f(x)- \f(y) | = d_{M}(x,y) \}  \setminus \{x=y \}  \Big).
\]

{\it Step 1.} Define the following map: $M_{2} \ni w \mapsto \Xi(w) : = \textrm{Argmin}\{ \f(w) - \f(z) : z\in M_{2} \}$.
Then the function $\f$ is a potential for the Monge minimization problem on $M$ with cost the geodesic distance, that coincides with $d_{M}$, 
with first marginal $m$ and as second marginal $\Xi_{\sharp} m$.

It follows from Proposition 15 of \cite{feldcann:mani} 
that $\nabla \f$ is a Lipschitz function: for all $x,y \in M_{2}$
\[
| \nabla \f (x) - \nabla \f (y)| \leq L d_{M}(x,y).
\] 
In \cite{feldcann:mani} the Lipschitz constant $L$ is uniform for $x,y$ belonging to sets uniformly far from the starting and ending points of the geodesics on $M$ of the transport set. Since in our setting the geodesics on $M$ do not intersect, $L$ is uniform on the whole $M$.
Moreover note that if $z = \gamma_{y}(t)$, then 
\[
\nabla \f(z) = - \frac{\dot{\gamma}_{y}(t) }{ |\dot{\gamma}_{y}(t)|}.
\]

{\it Step 2.} 
For $t \geq 0$, define the following map 
\[
M_{2} \ni x \mapsto \psi_{t} (x) : = x + \nabla \f (x) t.
\]
Possibly restricting $\psi_{t}$ to a subset of $M$ of points coming from transport rays of uniformly positive length,
since $ t \mapsto \psi_{t}(x)$ is a parametrization of the transport ray touching $M$ in $x$,
by $d_{M}$-cyclical monotonicity of $\Gamma$, we can assume that $\psi_{t}$ is injective.
Moreover $\psi_{t}$ is bi-Lipschitz, provided $t$ is small enough: indeed
\[
|x + \nabla \f (x) t - y - \nabla \f (y) t| \geq |x- y | (1 - L t).
\]
It follows that 
\[
M_{2} \times [-\delta,\delta] \ni (x,s) \mapsto \psi(x,s) : = x  + \nabla \f (x) (t + s)
\]
is  bi-Lipschitz and injective provided $\delta \leq  1/L + t$. Hence the Jacobian determinant of $\f$,  $J d \f $, is uniformly positive. 

{\it Step 3.}
Consider the following set
\[
B : =   \{ x \in \erre^{d} :  t - \delta \leq  d(M,x)  \leq t + \delta  \} \cap   G^{-1}(M)
\]
where $d$ is the euclidean distance. Clearly $B$ is the range of $\psi$ and $\mathcal{L}^{d}(B)> 0$. 
Since $M$ is a smooth manifold, we can pass to local charts: let $U_{\alpha} \subset \erre^{d-1}$ be an open set and  
$h_{\alpha} : U_{\alpha} \to M$ the corresponding parametrization map.
The map 
\[
U_{\alpha}\times [-\delta,\delta]\ni (x,s) \mapsto \psi_{\alpha}(x,s) : = \psi (h_{\alpha}(x),s)
\]
is a bi-Lipschitz parametrization of the set $B_{\alpha} : = B \cap G^{-1}( h_{\alpha}(U_{\alpha}))$. 

It follows directly from the Area Formula, see for example \cite{ambfuspal:bv}, that 
\[
\mathcal{L}^{d}\llcorner_{B_{\alpha}} = \psi_{\alpha\, \sharp} \Big(  J d \psi_{\alpha}( \mathcal{L}^{d-1}\times dt) \llcorner_{U_{\alpha} \times [-\delta,\delta]}   \Big), 
\]
hence $f^{M}_{\sharp} \mathcal{L}^{d}\llcorner_{B_{\alpha}} \ll \haus^{d-1} \llcorner_{M}$.
Since $B$ can be covered with a finite number of $B_{\alpha}$ and $\mathcal{L}^{d}\llcorner_{B_{\alpha}}$ is equivalent to $m$, the claim
follows.
\end{proof}

Recall the following result.
Let $(M,g)$ be a $n$-dimensional compact Riemannian manifold, let $d_{M}$ be the geodesic distance induced by $g$
and $\eta$ the volume measure.
Then the disintegration of $\eta$ w.r.t. any $d_{M}$-cyclically monotone set is strongly consistent and the conditional probabilities are continuous. 
This result is proved in \cite{biacava:streconv}, Theorem 9.5, in the more general setting of metric measure space satisfying the measure contraction property.

\begin{corollary}\label{C:boh}
For $m_{H}$-a.s. $y \in S$, the quotient measure $m_{\mu_{y}}$ is continuous.
\end{corollary}
\begin{proof}
We have proved in Remark \ref{R:tuttoM} that the measures $m_{\mu_{y}}$ are the conditional probabilities of the disintegration of $m$ w.r.t. the equivalence relation 
given by the membership to geodesics $\gamma_{y}$ and $m_{H}$ is the corresponding quotient measure. 
Hence the claim follows directly from Theorem 9.5 of \cite{biacava:streconv} and Proposition \ref{P:fondamentale}.
\end{proof}

\begin{proposition}\label{P:mappaostacolo1}
For $m_{H}$-a.e. $y \in S$, the measures
$\mu_{y,z}$ are continuous 
 for $m_{\mu_{y}}$-a.e. $z \in M$.
\end{proposition}

\begin{proof}
Recall that $f^{M}_{\sharp} \mu = m$.

{\it Step 1.} 
The measure $\mu$ can be disintegrated w.r.t. the partition given by the family of pre-images of the $\mathcal{A}$-measurable map $f^{M}$:
$\{ (f^{M})^{-1}(p) \}_{p \in f^{M}(\erre^{d})}$. Clearly $f^{M}$  is a possible quotient map, hence
\begin{equation}\label{E:conto}
\mu = \int \mu_{z} m(dz),
\end{equation}

The set $G^{-1}(M)\setminus a(M) \times G^{-1}(M)\setminus a(M) \cap G$ is $|\cdot |$-cyclically monotone and $\mu \ll \mathcal{L}^{d}$, 
hence it follows that for $m$-a.e. $z \in f^{M}(\erre^{d})$, $\mu_{z}$ is continuous.

{\it Step 2.} 
From Lemma \ref{L:appl} $\mu = \int \mu_{y} m_{H}(dy)$, therefore
\[
m = f^{M}_{\sharp}\mu = \int (f^{M}_{\sharp}\mu_{y}) m_{H}(dy),
\] 
hence using (\ref{E:conto}) and the uniqueness of the disintegration
\[
\mu = \int \bigg( \int \mu_{z} (f^{M}_{\sharp} \mu_{y})(dz)  \bigg) m_{H}(dy), \qquad \mu_{y} = \int \mu_{z} (f^{M}_{\sharp} \mu_{y})(dz), 
\]
where the last equality holds true for $m_{H}$-a.e. $y \in S$. 
Hence for $m_{H}$-a.e. $y \in S$ the measures $\mu_{y,z}$ are continuous for $m_{\mu_{y}}$-a.e. $z \in M$.
\end{proof}

Finally we can prove the existence of an optimal map for the Monge minimization problem with obstacle.

\begin{theorem}
There exists a solution for the Monge minimization problem with cost $d_{M}$ and marginal $\mu, \nu$ with $\mu \ll \mathcal{L}^{d}$. 
\end{theorem}

\begin{proof}
From Lemma \ref{L:appl} it follows that $\mu$ can be disintegrated w.r.t. the equivalence relation $H$.
From Theorem \ref{T:finale} it follows that to prove the claim it is enough to prove the existence of an optimal map on each equivalence class $H(y)$.
Hence we restrict the analysis to the classes $H(y)$ such that $H(y)\neq R(y)$ and for them we proved in 
Lemma \ref{L:geoH} that Assumption \ref{A:clessidra3} holds true. 
In Proposition \ref{P:fondamentale}, Corollary \ref{C:boh} and Proposition \ref{P:mappaostacolo1} we proved that 
for $m_{H}$-a.e. $y \in S$ the measures $m_{\mu_{y}}$ and $\mu_{y,z}$ verify the hypothesis of Proposition \ref{P:alternativa}. 
Therefore the claim follows.
\end{proof}

\medskip
\centerline{\sc Acknowledgements}
\vspace{2mm}
I would like to express my gratitude to Stefano Bianchini for his support during the preparation of this paper. 
I wish to thank Luigi De Pascale for drawing my attention on the obstacle problem.

\appendix

\section{Notation}
\label{S:notation}

\begin{tabbing}
\hspace{4cm}\=\kill
$P_{i_1\dots i_I}$ \> projection of $x \in \Pi_{k=1,\dots,K} X_k$ into its $(i_1,\dots,i_I)$ coordinates, keeping order
\\
$\mathcal{P}(X)$ or $\mathcal{P}(X,\Omega)$ \> probability measures on a measurable space $(X,\Omega)$
\\
$\mathcal{M}(X)$ or $\mathcal{M}(X,\Omega)$ \> signed measures on a measurable space $(X,\Omega)$
\\
$f \llcorner_A$ \> the restriction of the function $f$ to $A$
\\
$\mu \llcorner_A$ \> the restriction of the measure $\mu$ to the $\sigma$-algebra $A \cap \Sigma$
\\
$\mathcal{L}^d$ \> Lebesgue measure on $\R^d$
\\
$\mathcal{H}^k$ \> $k$-dimensional Hausdorff measure
\\
$\Pi(\mu_1,\dots,\mu_I)$ \> $\pi \in \mathcal{P}(\Pi_{i=1}^I X_i, \otimes_{i=1}^I \Sigma_i)$ with marginals $(P_i)_\sharp \pi = \mu_i \in \mathcal{P}(X_i)$
\\
$\mathcal{I}(\pi)$ \> cost functional (\ref{E:Ifunct})
\\
$c$ \> cost function $ : X \times Y \mapsto [0,+\infty]$
\\
$\mathcal{I}$ \> transportation cost (\ref{E:Ifunct})
\\
$\phi^c$ \> $c$-transform of a function $\phi$ (\ref{E:ctransf})
\\
$\partial^c \f$ \> $d$-subdifferential of $\f$ (\ref{E:csudiff})
\\
$\Phi_c$ \> subset of $L^1(\mu) \times L^1(\nu)$ defined in (\ref{E:Phicset})
\\
$J(\phi,\psi)$ \> functional defined in (\ref{E:Jfunct})
\\
$C_b$ or $C_b(X,\R)$ \> continuous bounded functions on a topological space $X$
\\
$(X,d)$ \> Polish space
\\
$(X,d_L)$ \> non-branching geodesic separable metric space
\\
$D_N(x)$ \> the set $\{y : d_N(x,y) < +\infty\}$
\\
$L(\gamma)$ \> length of the Lipschitz curve $\gamma$, Definition \ref{D:lengthstr}
\\
$B_r(x)$ \> open ball of center $x$ and radius $r$ in $(X,d)$
\\
$B_{r,L}(x)$ \> open ball of center $x$ and radius $r$ in $(X,d_L)$
\\
$\mathcal{K}(X)$ \> space of compact subsets of $X$
\\
$d_H(A,B)$ \> Hausdorff distance of $A$, $B$ w.r.t. the distance $d$
\\
$A_x$, $A^y$ \> $x$, $y$ section of $A \subset X \times Y$ (\ref{E:sectionxx})
\\
$\mathcal{B}$, $\mathcal{B}(X)$ \> Borel $\sigma$-algebra of $X$ Polish
\\
$\Sigma^1_1$, $\Sigma^1_1(X)$ \> the pointclass of analytic subsets of Polish space $X$, i.e.~projection of Borel sets
\\
$\Pi^1_1$ \> the pointclass of coanalytic sets, i.e.~complementary of $\Sigma^1_1$
\\
$\Sigma^1_n$, $\Pi^1_n$ \> the pointclass of projections of $\Pi^1_{n-1}$-sets, its complementary
\\
$\Delta^1_n$ \> the ambiguous class $\Sigma^1_n \cap \Pi^1_n$
\\
$\mathcal{A}$ \> $\sigma$-algebra generated by $\Sigma^{1}_{1}$
\\
$\mathcal{A}$-function \> $f : X \to \R$ such that $f^{-1}((t,+\infty])$ belongs to $\mathcal A$
\\
$h_\sharp \mu$ \> push forward of the measure $\mu$ through $h$, $h_\sharp \mu(A) = \mu(h^{-1}(A))$
\\
$\textrm{graph}(F)$ \> graph of a multifunction $F$ (\ref{E:graphF})
\\
$F^{-1}$ \> inverse image of multifunction $F$ (\ref{E:inverseF})
\\
$F_x$, $F^y$ \> sections of the multifunction $F$ (\ref{E:sectionxx})
\\
$\mathrm{Lip}_1(X)$ \> Lipschitz functions with Lipschitz constant $1$
\\
$\Gamma'$ \> transport set (\ref{E:gGamma})
\\
$G$, $G^{-1}$ \> outgoing, incoming transport ray, Definition \ref{D:Gray}
\\
$R$ \> set of transport rays (\ref{E:Rray})
\\
$\mathcal{T}$, $\mathcal{T}_e$ \> transport sets (\ref{E:TR0})
\\
$a,b : \mathcal{T}_e \to \mathcal{T}_e$ \> endpoint maps (\ref{E:endpoint0})
\\
$\mathcal{Z}_{m,e}$, $\mathcal Z_m$ \> partition of the transport set $\Gamma$ (\ref{E:Zkije}), (\ref{E:mapTijkF})
\\
$\mathcal S$ \> cross-section of $R \llcorner_{\mathcal T \times \mathcal T}$
\\
$A_t$ \> evolution of $A \subset \mathcal{Z}_{k,i,j}$ along geodesics (\ref{E:At})

\end{tabbing}

\bibliography{biblio}

\end{document}